\documentclass{amsart}
\usepackage{amsfonts,amssymb,amsmath,amsthm}
\usepackage{url}
\usepackage{color}
\usepackage{enumerate}
\usepackage{bm}

\newtheorem{thm}{Theorem}[section]
\newtheorem{cor}[thm]{Corollary}
\newtheorem{lem}[thm]{Lemma}
\newtheorem{prop}[thm]{Proposition}

\theoremstyle{definition}
\newtheorem{defin}[thm]{Definition}
\newtheorem{rem}[thm]{Remark}
\newtheorem{exa}[thm]{Example}

\numberwithin{equation}{section}

\makeatletter
\@namedef{subjclassname@2020}{%
  \textup{2020} Mathematics Subject Classification}
\makeatother

\author[A. Itaba]{Ayako Itaba}
\address{
Department of Mathematics, 
Faculty of Science, Tokyo University of Science\\
1-3 Kagurazaka, Shinjyuku-ku, Tokyo, 162-8601, Japan}
\email{itaba@rs.tus.ac.jp}

\author[M. Matsuno]{Masaki Matsuno}
\address{
Graduate School of Integrated Science and Technology, 
Shizuoka University\\
Ohya 836, Shizuoka 422-8529, Japan}
\email{matsuno.masaki.14@shizuoka.ac.jp}


\DeclareMathOperator{\Pic}{Pic}
\DeclareMathOperator{\Div}{Div}

\newcommand{\la}{\lambda}

\begin{document}

\title[AS-regularity of geometric algebras of plane cubic curves]
{AS-regularity of geometric algebras of plane cubic curves}
\begin{abstract}
In noncommutative algebraic geometry, 
an Artin-Schelter regular (AS-regular) algebra  
is one of the main interests 
and every $3$-dimensional quadratic AS-regular algebra is a 
geometric algebra introduced by Mori whose point scheme 
is either $\mathbb{P}^{2}$ or a cubic curve in $\mathbb{P}^{2}$ 
by Artin-Tate-Van den Bergh.
In the preceding paper 
by the authors, 
we determined the all possible defining relations 
for these geometric algebras. 
However, the authors did not check the AS-regularity 
of these geometric algebras.
In this paper, by using twisted superpotentials and 
twists of superpotentials in the sense of Mori-Smith, 
we check the AS-regularity of geometric algebras  
whose point schemes are not elliptic curves. 
For geometric algebras whose point schemes are elliptic curves, 
we give a simple condition to be $3$-dimensional quadratic AS-regular algebras. 
As an application, 
we show that every $3$-dimensional quadratic AS-regular algebra 
is graded Morita equivalent to a Calabi-Yau AS-regular algebra. 
\end{abstract}
\maketitle
\section{Introduction}
In noncommutative algebraic geometry, 
an Artin-Schelter regular (AS-regular) algebra 
introduced by Artin-Schelter \cite{AS} is one of the main interests.
Artin-Tate-Van den Bergh \cite{ATV1} 
proved that there exists a one-to-one correspondence 
between $3$-dimensional AS-regular algebras  and regular geometric pairs. 
This work convinced us that algebraic geometry is very useful 
to study even noncommutative algebras. 
Dubois-Violette \cite{D} and Bocklandt-Schedler-Wemyss \cite{BSW}
showed that every $3$-dimensional quadratic AS-regular algebra $A$ 
is isomorphic to a derivation-quotient algebra $\mathcal{D}(w)$ 
of a twisted superpotential $w$, 
and Mori-Smith \cite{MS1} showed that such 
$w$ is unique up to non-zero scalar multiples. 
So it is interesting to study AS-regular algebras using 
both algebraic geometry and twisted superpotentials. 
In fact, Mori-Smith \cite{MS2} classified $3$-dimensional quadratic 
Calabi-Yau AS-regular algebras by using superpotentials.

In the preceding paper \cite{IM} by the authors, 
in terms of geometric algebras defined by Mori \cite{Mo}, 
we determined all possible defining relations 
for geometric algebras whose point schemes are
either $\mathbb{P}^{2}$ or cubic curves in $\mathbb{P}^{2}$
and classified them up to graded algebra isomorphism 
and up to graded Morita equivalence. 
However, in \cite{IM}, 
the authors did not check the AS-regularity 
of these classified geometric algebras.
So, one of the aims of this paper is to check the AS-regularity of them. 
Note that, Iyudu-Shkarin \cite{IS} recently 
gave a list of defining relations 
of $3$-dimensional AS-regular algebras by using twisted superpotentials, 
but no proof of AS-regularity of these algebras
was given in their paper. 
For geometric algebras listed in \cite[Theorem 3.1]{IM},
we give a list of candidates of twisted superpotentials
to serve our purposes (see Proposition \ref{TSP}). 
By using this list, we give a complete list of superpotentials
whose derivation-quotient algebras are $3$-dimensional quadratic
Calabi-Yau AS-regular algebras
whose point schemes are not elliptic curves 
(see Theorem \ref{SP}).
By using {\it a twist of a superpotential} (in the sense of \cite{MS1}), 
we showed that potentials listed in Proposition \ref{TSP}
are in fact twisted superpotentials and derivation-quotient algebras of them
are $3$-dimensional quadratic AS-regular algebras
(see Theorems \ref{prop-3} and \ref{thm-RTSP}). 
For a geometric algebra $A$ whose point scheme is an elliptic curve 
in $\mathbb{P}^{2}$, 
we give a simple condition that $A$ is AS-regular 
(see Theorem \ref{Main}).
As an application of Corollary \ref{cor_conj} and Theorem \ref{Main}, 
we prove the following theorem (see Theorem \ref{Main2}). 
\begin{thm}
\label{conj}
For every $3$-dimensional quadratic AS-regular algebra $A$, 
  there exists a Calabi-Yau AS-regular algebra $S$ 
  such that $A$ and $S$ are graded Morita equivalent. 
\end{thm}
\noindent
Theorem \ref{conj} tells us that, 
for a $3$-dimensional quadratic AS-regular algebra $A$, 
to study the noncommutative projective scheme ${\rm Proj}_{\rm nc}A$ of $A$ 
in the sense of  Artin-Zhang \cite{AZ} 
is reduced to study ${\rm Proj}_{\rm nc}S$ 
for the Calabi-Yau AS-regular algebra $S$. 
Note that \cite[Example 14]{Ueyama2} gave 
one example of a $3$-dimensional cubic AS-regular algebra 
which is not graded Morita equivalent 
to any Calabi-Yau AS-regular algebra. 

This paper is organized as follows: 
In Section 2, 
we recall the definition of an AS-regular algebra defined 
by Artin-Schelter \cite{AS}, a Calabi-Yau algebra 
by Ginzburg \cite{G}, a twisted superpotential and a twist of 
a superpotential in the sense of \cite{MS1}.
Also, we recall Zhang's twist and twisted algebras from \cite{Z} 
and some lemmas which are needed to show our Theorem \ref{conj}. 
Moreover, we recall the definitions of  a geometric algebra 
for quadratic algebras introduced by Mori \cite{Mo}, and 
the result of our preceding paper \cite{IM}. 
In Section 3, we prove Theorem \ref{conj} 
for geometric algebras 
whose point schemes are not elliptic curves. 
Finally, in Section 4, we prove Theorem \ref{conj} 
for geometric algebras whose point schemes 
are elliptic curves in $\mathbb{P}^{2}$. 
\section{Preliminaries}
Throughout this paper, 
let $k$ be an algebraically closed field of 
characteristic $0$. 
A graded $k$-algebra means an $\mathbb{N}$-graded algebra
$A=\bigoplus_{i\in \mathbb{N}}A_{i}$. 
A {\it connected graded} $k$-algebra $A$ is 
a graded $k$-algebra such that $A_{0}=k$. 
We denote by ${\rm GrMod}\,A$ the category of graded right $A$-modules 
and graded right $A$-module homomorphisms and 
we say that two graded $k$-algebras $A$ and $B$ are graded Morita equivalent 
if two categories ${\rm GrMod}\,A$ and ${\rm GrMod}\,B$ are equivalent. 
\subsection{AS-regular algebras and Calabi-Yau algebras}
Let $A$ be 
a connected graded  $k$-algebra finitely generated by elements of positive degree. 
 We recall that 
 $${\rm GKdim}\,A:=
 {\rm inf}\{\alpha\in \mathbb{R} \mid {\rm dim}_{k}(\sum_{i=0}^{n}A_{i})\leq n^{\alpha}
 \text{ for all }n \gg 0\}
 $$
 is called the {\it Gelfand-Kirillov dimension} of $A$. 
\begin{defin}[{\cite[page 171]{AS}}]
\label{AS-reg}
A connected graded $k$-algebra $A$ is called 
a {\it $d$-dimensional Artin-Schelter regular 
{\rm (}AS-regular{\rm )} algebra} 
if $A$ satisfies the following conditions: 
\begin{enumerate}[(i)]
 \item ${\rm gldim}\,A=d<\infty$, 
 \item ${\rm GKdim}\,A<\infty$, 
 \item ({\it Gorenstein condition})\quad
       ${\rm Ext}_{A}^{i}(k,A)\cong \left\{
                                         \begin{array}{ll}
                                         k   &\quad (i=d), \\
                                         0   &\quad (i\neq d). 
                                         \end{array}
                                         \right.$
\end{enumerate}
\end{defin}
Any $3$-dimensional AS-regular algebra $A$ 
finitely generated in degree $1$ 
is a graded algebra isomorphic to an algebra of the form 
\[
k\langle x, y, z \rangle/(f_{1},f_{2},f_{3})\text{ (quadratic case) } ,
\ \text{ or }\ k\langle x, y\rangle/(g_{1}, g_{2})\text{ (cubic case) }
\]
where $f_{i}$ are homogeneous polynomials of degree $2$ 
and $g_{i}$ are homogeneous polynomials of degree $3$ 
(\cite[Theorem 1.5 (i)]{AS}). 
In this paper, 
we focus on $3$-dimensional quadratic AS-regular algebras. 

Let $V$ be a $3$-dimensional $k$-vector space and 
$T(V)$ the tensor algebra of $V$. 
We choose a basis $\{x_{1},x_{2},x_{3}\}$ of $V$. 
Also, for an algebra $T(V)/(R)$,
we choose a basis $\{f_{1},f_{2},f_{3}\}$ of $R \subset V^{\otimes 2}$. 
We set $\bm{x}:=(x_{1},x_{2},x_{3})^{t}$ and 
$\bm{f}:=(f_{1},f_{2},f_{3})^{t}$, 
where, for a matrix $N$, $N^{t}$ means the transpose of $N$. 
There is a unique $3\times 3$ matrix $M$ with entries in $V$ 
such that $\bm{f}=M\bm{x}$ (see {\cite[page 177]{AS}). 
From \cite[page 34]{ATV1}, $T(V)/(R)$ is called {\it standard} 
if there are bases for $V$ and $R$ such that 
the entries in $\bm{x}^{t}M$ are also a basis for $R$. 
\begin{thm}[{\cite[Theorem 1]{ATV1}}] 
\label{ATV1_thm1}
Let $V$ be a $3$-dimensional $k$-vector space and 
$R$ a $3$-dimensional subspace of $V^{\otimes 2}$. 
Then $T(V)/(R)$ is a $3$-dimensional AS-regular algebra 
if and only if $T(V)/(R)$ is standard and 
the common zero locus in $\mathbb{P}^{2}$ of the $2\times 2$ minors 
of the matrix $M$ in the above is empty. 
\end{thm}
Here, we recall the definition of a Calabi-Yau algebra introduced by \cite{G}. 
\begin{defin}[{\cite[Definition 3.2.3]{G}}]
A $k$-algebra $S$ is called {\it $d$-dimensional Calabi-Yau} 
if $S$ satisfies the following conditions: 
\begin{enumerate}[(i)]
\item $\mathrm{pd}_{S^{\rm e}}S=d<\infty$, 
\item $
       \mathrm{Ext}_{S^{\rm e}}^{i}(S,S^{\rm {e}})\cong
       \begin{cases}
        S & \text{ if } i=d, \\
        0 & \text{ if } i \neq d
       \end{cases}
       $ \quad (as right $S^{\rm e}$-modules)
\end{enumerate}
where $S^{\rm e}=S^{{\rm op}} \otimes_{k} S$ 
is the enveloping algebra of $S$. 
\end{defin}
For example, 
it is known that 
an $n$-th polynomial ring $k[x_{1},x_{2},\ldots,x_{n}]$ is 
$n$-dimensional Calabi-Yau. 

Note that, for a $3$-dimensional quadratic AS-regular algebra $A$, 
the quadratic dual $A^{!}$ of $A$ 
is a Frobenius algebra by \cite[Proposition 5.10]{S}. 
Hence, we can consider the Nakayama automorphism $\nu_{A^{!}}$ of $A^{!}$. 
By using the following consequence proved 
by Reyes-Rogalski-Zhang \cite{RRZ}, 
we can determine whether these algebras $A$ are
Calabi-Yau algebras or not. 
\begin{lem}[{\cite[comments after the proof of Example 1.4]{RRZ}}]
\label{thmRRZ}
Let $A$ be a $3$-dimensional quadratic AS-regular algebra. 
Then $A$ is Calabi-Yau if and only if 
the Nakayama automorphism $\nu_{A^{!}}$ of $A^{!}$ is the identity 
{\rm (}that is, $A^{!}$ is symmetric{\rm )}. 
\end{lem}
\subsection{Twisted algebras}
In this subsection, 
we recall the notion of twisting system and twisted algebra
introduced by Zhang \cite{Z}. 

A set of graded $k$-linear automorphisms of $A$, 
say $\theta=\{\theta_{i}\mid i \in \mathbb{N}\}$, 
is called a {\it twisting system} of $A$ 
if 
$$
\theta_{n}(a\theta_{l}(b))=\theta_{n}(a)\theta_{n+l}(b)
$$ 
for all $l,m,n \in \mathbb{N}$ and all $a \in A_{l} , b\in A_{m}$ 
(\cite[Definition 2.1]{Z}). 
Let $\theta=\{\theta_{i}\mid i \in \mathbb{N}\}$ 
be a twisting system of $A$. 
Then, a new graded and associative multiplication 
$\ast$ on the underlying graded $k$-vector space
$\bigoplus_{i\in \mathbb{N}}A_{i}$ 
is defined by 
$$
a\ast b:=a \theta_{l}(b)\quad\text{for all }a \in A_{l}, b\in A_{m}. 
$$
We denote by $1_{\theta}$ the identity with respect to $\ast$. 
The graded $k$-algebra $(\bigoplus_{i\in \mathbb{N}}A_{i},\ast,1_{\theta})$ 
is called the {\it twisted algebra} of $A$ 
by $\theta$ and is denoted by $A^{\theta}$ 
(\cite[Definition 2.3]{Z}). 
Any graded algebra automorphism $\theta \in {\rm Aut}\,A$ 
defines a twisting system of $A$ by 
$\{ \theta^{i} \}_{i \in \mathbb{N} }$. 
The twist of $A$ by this twisting system is denoted by $A^{\theta}$ 
instead of $A^{\{\theta^{i} \}_{i \in \mathbb{N}}}$ 
which is called the {\it twist} of $A$ by $\theta$. 
\begin{lem}[{\cite[Theorem 3.5]{Z}}]
\label{twistingequi}
Let $A$ and $A'$ be two connected graded $k$-algebras 
with $A_{1}\neq 0$.
Then $A'$ is isomorphic to a twisted algebra of $A$ 
if and only if 
${\rm GrMod}\,A$ and ${\rm GrMod}\,A'$ 
are equivalent. 
\end{lem}
\begin{lem}[{\cite[Theorem 5.11 (b)]{Z}}]
\label{twist_inv}
Let $A$ be a connected graded $k$-algebra and 
$\theta$ be a twisting system of $A$. 
Then $A$ is a $3$-dimensional quadratic AS-regular algebra 
if and only if 
the twist $A^{\theta}$ of $A$ by $\theta$ 
is also a $3$-dimensional quadratic AS-regular algebra. 
\end{lem}
\subsection{Derivation-quotient algebras}
Now, we recall the definitions of superpotentials, 
twisted superpotentials and derivation-quotient algebras from 
\cite{BSW} and \cite{MS1}. 
Also, we recall the definition of a {\it twist of a superpotential} 
due to \cite{MS1} (see Definition \ref{MStwist}). 

Fix a basis $\{ x_{1},x_{2},x_{3} \}$ for $V$.
For $w\in V^{\otimes 3}$, 
there exist unique $w_{i}\in V^{\otimes 2}$ such that 
$w=\sum_{i=1}^{3}x_{i}\otimes w_{i}$. 
Then the {\it partial derivative} of $w$ with respect to $x_{i}$ 
($i=1,2,3$) is
$
\partial_{x_{i}}(w):=w_{i},
$ 
and the {\it derivation-quotient algebra} of $w$ is
$$
\mathcal{D}(w):=T(V)/(\partial_{x_{1}}w,\partial_{x_{2}}w,\partial_{x_{3}}w).
$$
Note that we call an element $w \in V^{\otimes 3}$ a {\it potential} 
in this paper.
We define the $k$-linear map $\varphi$: 
$V^{\otimes 3}\longrightarrow V^{\otimes 3}$ by 
$
\varphi(v_{1}\otimes v_{2}\otimes v_{3}):=v_{3}\otimes v_{1}\otimes v_{2}
$. 
We write ${\rm GL}(V)$ for the general linear group of $V$. 
\begin{defin}[{\cite[Introduction]{BSW}, \cite[Definition 2.5]{MS1}}]
\label{def_sp}
	Let $w \in V^{\otimes 3}$ be a potential.
\begin{enumerate}
\item If $\varphi(w)=w$,
           then $w$ is called a {\it superpotential}. 
\item
If there exists $\theta\in {\rm GL}(V)$ such that 
           $$
           (\theta\otimes {\rm id}\otimes {\rm id})\varphi(w)=w,
           $$
           then $w$ is called a {\it twisted superpotential}.
\end{enumerate}
\end{defin}
\begin{rem}
By Dubois-Violette \cite{D} and Bocklandt-Schedler-Wemyss \cite{BSW},
every $3$-dimensional quadratic AS-regular algebra $A$ 
is isomorphic to a derivation-quotient algebra $\mathcal{D}(w)$ 
of a twisted superpotential $w$ (see \cite[Theorem 5]{D} and \cite[Theorem 6.8]{BSW}), 
and by Mori-Smith \cite{MS1}, 
such $w$ is unique up to non-zero scalar multiples 
(see \cite[Proposition 2.12]{MS1}).
\end{rem}
\begin{defin}[{\cite[page 390]{MS1}}]
\label{MStwist}
For a superpotential $w\in V^{\otimes 3}$ and 
$\theta \in {\rm GL}(V)$,  
$$
w^{\theta}:=(\theta^{2}\otimes \theta\otimes {\rm id})(w) 
$$is called a {\it Mori-Smith twist} ({\it MS twist}) of $w$ by $\theta$. 
\end{defin}

For a potential $w \in V^{\otimes 3}$, 
we set 
$$
{\rm Aut}\,(w):=
\{\theta \in {\rm GL}(V) \mid  
(\theta^{\otimes 3})(w)
=\lambda w,\,\exists \lambda \in k \setminus \{0\}\}. 
$$
For any $w \in V^{\otimes 3}$, it follows from \cite[Lemma 3.1]{MS1} that
${\rm Aut}\,(w) \subset {\rm Aut}\,\mathcal{D}(w)$.
\begin{lem}
	[{\cite[Proposition 5.2]{MS1}}]
	\label{twist}
	For a superpotential $w\in V^{\otimes 3}$ and 
	$\theta \in {\rm Aut}\,(w)$,
	we have that $\mathcal{D}(w^{\theta}) \cong \mathcal{D}(w)^{\theta}$.
\end{lem}
\begin{lem}
\label{lem-1}
If $w \in V^{\otimes 3}$ is a superpotential and 
$
\theta \in
{\rm Aut}\,(w)
$, 
then the MS twist $w^{\theta}$ of a superpotential $w$ by $\theta$ 
is a twisted superpotential. 
\end{lem}
\begin{proof}
Let $w \in V^{\otimes 3}$ be a superpotential 
and $\theta \in {\rm Aut}\,(w)$.
By definition, 
there exists $\lambda \in k \setminus \{0\}$ 
such that $(\theta^{\otimes 3})(w)=\lambda w$.
We set $\theta':=\lambda^{-1}\theta^{3} \in {\rm GL}(V)$.
Since $w$ is a superpotential,
\begin{align*}
(\theta' \otimes {\rm id} \otimes {\rm id})(\varphi(w^{\theta}))
			&=(\theta' \otimes {\rm id} \otimes {\rm id})({\rm id} \otimes \theta^{2} \otimes \theta)(\varphi(w)) \\
			&=\lambda^{-1}(\theta^{3} \otimes \theta^{2} \otimes \theta)(w) 
			=\lambda^{-1}(\theta^{2} \otimes \theta \otimes {\rm id})(\lambda w)
			=w^{\theta},
\end{align*}
so the MS twist $w^{\theta}$ is a twisted superpotential.
\end{proof}
\begin{rem}
If $\theta \in {\rm GL}(V)\setminus{\rm Aut}\,(w)$, 
then the MS twist $w^{\theta}$ of a superpotential $w\in V^{\otimes 3}$ 
by $\theta$ may not be a twisted superpotential. 
Indeed, let $w:=x^{3}\in V^{\otimes 3}$. 
Since $\varphi(w)=w$, we see that $w$ is a superpotential. 
Take 
$
\theta:=\left(
\begin{array}{ccc}
0 &1 & 0 \\
1 & 0& 0 \\
0 & 0& 1
\end{array}
\right) \in {\rm GL}_{3}(k)
$. 
Then 
$
w^{\theta}=(\theta^{2}\otimes \theta \otimes {\rm id})(w)=xyx. 
$
Since, for any $\theta' \in {\rm GL}(V)$, 
$
(\theta'\otimes {\rm id}\otimes {\rm id})(\varphi(w^{\theta}))
=\theta'(x)xy
\neq w^{\theta}
$, 
the MS twist $w^{\theta}$ is not a twisted superpotential. 
Note that 
we have $\theta \not\in {\rm Aut}\,(w)$ 
by 
$
(\theta^{\otimes 3})(w)=y^{3}\neq w
$. 
\end{rem}
\begin{defin}\label{de-pot}
Let $w \in V^{\otimes 3}$ be a potential.
\begin{enumerate}[{(1)}]
\item A potential $w$ is called {\it regular} 
          if the derivation-quotient algebra $\mathcal{D}(w)$ is
          a $3$-dimensional quadratic AS-regular algebra.
\item A potential $w$ is called {\it Calabi-Yau} 
          if the derivation-quotient algebra $\mathcal{D}(w)$ is
          a $3$-dimensional Calabi-Yau AS-regular algebra.
\end{enumerate}
\end{defin}
\begin{rem}
\label{Bockdland}
By Bockdlandt \cite{B}, 
every $3$-dimensional quadratic Calabi-Yau AS-regular algebra 
is isomorphic to a derivation-quotient algebra $\mathcal{D}(w)$ of 
a superpotential $w$ (\cite[Theorem 3.1]{B}). 
\end{rem}
\begin{lem}[{\cite[Corollary 4.5]{MS1}}]\label{CY-superpotential}
	Let $w \in V^{\otimes 3}$ be regular.
	Then $w$ is Calabi-Yau if and only if it is a superpotential.
\end{lem}
\begin{lem}
\label{lem-2}
If $w$ is a Calabi-Yau superpotential and $\theta \in {\rm Aut}\,(w)$, 
then the MS twist $w^{\theta}$ 
of a superpotential $w$ by $\theta$ is a regular twisted superpotential. 
\end{lem}
\begin{proof}
Let $w \in V^{\otimes 3}$ be a Calabi-Yau superpotential 
and $\theta\in {\rm Aut}\,(w)$.
By Lemma \ref{lem-1}, 
the MS twist $w^{\theta}$ of $w$ by $\theta$ is a twisted superpotential, 
so it is sufficient to show that $w^{\theta}$ is regular.
Since $w$ is Calabi-Yau, $\mathcal{D}(w)$ is Calabi-Yau AS-regular, and
since $\theta \in {\rm Aut}\,(w)$, by Lemma \ref{twist}, we have that
$\mathcal{D}(w^{\theta}) \cong \mathcal{D}(w)^{\theta}$.
Since it holds from Lemma \ref{twist_inv}
that AS-regularity is presented by twisting, 
$\mathcal{D}(w^{\theta})$ is AS-regular, 
that is, the MS twist $w^{\theta}$ is regular.
\end{proof}
\begin{exa}
\label{ex_S1}
We set 
$
w:=(xyz+yzx+zxy)-(zyx+yxz+xzy)\in V^{\otimes 3}
$. 
Then, we see that 
$\varphi(w)=w
$. 
So, $w$ is a superpotential. 
The derivation-quotient algebra of $w$ is
\begin{align*}
\mathcal{D}(w)
&=k\langle x,y,z\rangle/(\partial_{x}w,\partial_{y}w,\partial_{z}w)\\
&=k\langle x,y,z\rangle/(yz-zy,zx-xz,xy-yx)=k[x,y,z]. 
\end{align*}
It is known that a polynomial ring $k[x,y,z]$ is Calabi-Yau AS-regular. 
So, by Definition \ref{de-pot} (2),
 $w$ is a Calabi-Yau superpotential. 
Take  
$
\theta:=\left(
\begin{array}{ccc}
\alpha &0 & 0 \\
0 & \beta& 0 \\
0 & 0& \gamma
\end{array}
\right)\in {\rm GL}_{3}(k)
$. 
Calculating the MS twist $w^{\theta}$ of the superpotential $w$ by $\theta$, 
\begin{align*}
w^{\theta}&=(\theta^{2}\otimes \theta\otimes {\rm id})(w)\\
&=(\alpha^{2}\beta xyz+\beta^{2}\gamma yzx+\alpha\gamma^{2}zxy)
-(\beta\gamma^{2}zyx+\alpha\beta^{2}yxz+\alpha^{2}\gamma xzy). 
\end{align*}
Therefore, the derivation-quotient algebra of $w^{\theta}$ 
is as follows: 
\begin{align*}
\mathcal{D}(w^{\theta})
&=k\langle x,y,z\rangle/(\partial_{x}w,\partial_{y}w,\partial_{z}w)\\
&=k\langle x,y,z\rangle/
(\alpha^{2}\beta yz-\alpha^{2}\gamma zy,
\beta^{2}\gamma zx-\alpha\beta^{2}xz,
\alpha\gamma^{2}xy-\beta\gamma^{2}yx)\\
&=k\langle x,y,z\rangle/
(\beta yz-\gamma zy,
\gamma zx-\alpha xz,
\alpha xy-\beta yx). 
\end{align*}
Since ${\rm Aut}\,(w)={\rm GL}\,(V)$, 
we see that $\theta \in {\rm Aut}\,(w)$. 
By Lemma \ref{lem-2}, $w^{\theta}$ is a regular twisted superpotential, 
so, $\mathcal{D}(w^{\theta})$ is an AS-regular algebra. 
\end{exa}
\subsection{Geometric algebras}
Let $(E,\mathcal{O}_{E})$ be a scheme
where $\mathcal{O}_{E}$ is the structure sheaf on $E$.
An {\it invertible sheaf} on $E$ is defined to be a
locally free $\mathcal{O}_{E}$-module of rank $1$.
For a quadratic algebra $A=T(V)/(R)$, 
we set 
$$
\mathcal{V}(R):=\{(p,q)\in \mathbb{P}(V^{\ast})\times \mathbb{P}(V^{\ast})
\mid f(p,q)=0 \text{ for all } f\in R\}. 
$$
Let $E \subset \mathbb{P}(V^{\ast})$ be a closed
$k$-subscheme and $\sigma$ an automorphism of $E$.
For the rest of the paper, we fix
\begin{enumerate}[(a)]
	\item $\pi: E \rightarrow \mathbb{P}(V^{\ast})$ is the embedding,
	\item $\mathcal{L}:=\pi^{\ast}(\mathcal{O}_{\mathbb{P}(V^{\ast})}(1))$.
\end{enumerate}
In this case, $\mathcal{L}$ becomes an invertible sheaf on $E$. 
The map
$$
\mu: {\rm H}^{0}(E,\mathcal{L}) \otimes {\rm H}^{0}(E,\mathcal{L})
\rightarrow
{\rm H}^{0}(E,\mathcal{L}) \otimes {\rm H}^{0}(E,\mathcal{L}^{\sigma})
\rightarrow
{\rm H}^{0}(E,\mathcal{L}\otimes_{\mathcal{O}_{E}} \mathcal{L}^{\sigma})
$$
of $k$-vector spaces 
is defined by
$v \otimes w \mapsto v \otimes w^{\sigma}$
where $\mathcal{L}^{\sigma}=\sigma^{\ast}\mathcal{L}$
and $w^{\sigma}=w \circ \sigma$. 

For a quadratic algebra, a {\it geometric algebra} was introduced by 
Mori \cite{Mo}. 
\begin{defin}[{\cite[Definition 4.3]{Mo}}]
	A quadratic algebra $A=T(V)/(R)$ is called {\it geometric} 
        if there is a pair $(E,\sigma)$
	where $E \subset \mathbb{P}(V^{\ast})$ is a closed $k$-subscheme, 
        and $\sigma$ is a $k$-automorphism of $E$ such that
	\begin{itemize}
		\item (G1): $\mathcal{V}(R)=\{
		(p,\sigma(p)) \in \mathbb{P}(V^{\ast}) \times \mathbb{P}(V^{\ast}) \mid p \in E
		\}$, and
		\item (G2): $R={\rm ker} \mu$ with the identification
		$$
		{\rm H}^{0}(E,\mathcal{L})
		={\rm H}^{0}(\mathbb{P}(V^{\ast}),\mathcal{O}_{\mathbb{P}(V^{\ast})}(1))
                =V\text{ as $k$-vector spaces.}
		$$
	\end{itemize}
	When $A$ satisfies the condition (G2), we write $A=\mathcal{A}(E,\sigma)$.
	
\end{defin}
Let $A=T(V)/(R)$ be a quadratic algebra.
If $A=\mathcal{A}(E,\sigma)$ is a geometric algebra, 
then $E$ is called the point scheme of $A$.
If $E$ is reduced,
then the condition (G2) is equivalent to the condition (G2'):
$R=\{
f \in V \otimes V \mid f|_{\mathcal{V}(R)}=0
\}$ (see \cite{Mo}).
\begin{thm}[{\cite[Theorem 3]{ATV1}}]
	\label{ATV1}
	Let $A$ be a quadratic algebra. 
	Then $A$ is a $3$-dimensional AS-regular algebra if and only if
	$A$ is isomorphic to a geometric algebra $\mathcal{A}(E,\sigma)$ 
	which satisfies one of the following conditions{\rm :}
	\begin{enumerate}[{\rm (1)}]
		\item $E=\mathbb{P}^{2}$ and $\sigma \in {\rm Aut}_{k}\,\mathbb{P}^{2}$. 
		\item $E$ is a cubic curve in $\mathbb{P}^{2}$ and $\sigma \in {\rm Aut}_{k}\,E$
		such that $\sigma^{\ast}\mathcal{L} \not\cong \mathcal{L}$ and
   	$$
   	(\sigma^{2})^{\ast}\mathcal{L} \otimes_{\mathcal{O}_{E}} \mathcal{L}
		\cong \sigma^{\ast}\mathcal{L} \otimes_{\mathcal{O}_{E}} \sigma^{\ast}\mathcal{L}.
		$$
	\end{enumerate}
\end{thm}
The types of $(E,\sigma)$ of $3$-dimensional quadratic 
AS-regular algebras are defined in \cite{MU1} 
which are slightly modified from the original types 
defined in \cite{AS} and \cite{ATV1}. 
We extend the types defined in \cite{MU1} as follows 
(see \cite[Subsection 2.3]{IM}): 
\begin{description}
\item[{\rm (1) Type P}] 
      $E$ is $\mathbb{P}^{2}$, 
      and $\sigma \in \mathrm{Aut}_{k}\mathbb{P}^{2}=\mathrm{PGL}_{3}(k)$ 
      (Type  $\mathbb{P}^{2}$ is divided into 
      Type P$_{i}$ ($i=1,2,3$) in terms of the Jordan canonical form of $\sigma$). 
\item [{\rm (2-1) Type S$_{1}$}] 
      $E$ is a triangle, and $\sigma$ stabilizes each component. 
\item[{\rm (2-2) Type S$_{2}$}]
      $E$ is a triangle, and $\sigma$ interchanges two of its components. 
\item[{\rm (2-3) Type S$_{3}$}]
      $E$ is a triangle, and $\sigma$ circulates three components. 
\item[{\rm (3-1) Type S'$_{1}$}] 
      $E$ is a union of a line and a conic meeting at two points, 
      and $\sigma$ stabilizes each component 
      and two intersection points.
\item[{\rm (3-2) Type S'$_{2}$}]
      $E$ is a union of a line and a conic meeting at two points, 
      and $\sigma$ stabilizes each component 
      and interchanges two intersection points.
\item[{\rm (4-1) Type T$_{1}$}]
      $E$ is a union of three lines meeting at one point, 
      and $\sigma$ stabilizes each component.
\item[{\rm (4-2) Type T$_{2}$}]
      $E$ is a union of three lines meeting at one point, 
      and $\sigma$ interchanges two of its components. 
\item[{\rm (4-3) Type T$_{3}$}]
      $E$ is a union of three lines meeting at one point, 
      and $\sigma$ circulates three components.
\item[{\rm (5) Type T'}]
      $E$ is a union of a line and a conic meeting at one point, 
      and $\sigma$ stabilizes each component. 
\item[{\rm (6) Type CC}] $E$ is a cuspidal cubic curve. 
\item[{\rm (7) Type NC}]
          $E$ is a nodal cubic curve 
          (Type NC is divided into Type NC$_{i}$ $(i=1,2)$). 
\item[{\rm (8) Type WL}]$E$ is a union of a double line and a line 
          (Type WL is divided into Type WL$_{i}$ $(i=1,2,3)$). 
\item[{\rm (9) Type TL}] $E$ is a triple line 
         (Type TL is divided into Type TL$_{i}$ $(i=1,2,3,4)$). 
\item[{\rm (10) Type EC}] $E$ is an elliptic curve. 
\end{description}
\begin{rem}
\label{rem_IM1}
All possible defining relations of 
$3$-dimensional quadratic AS-regular algebras 
are listed in each type up to isomorphism 
from (1) through (9) in \cite[Theorem 3.1]{IM}, 
and (10) in \cite[Theorem 4.9]{IM}. 
\end{rem}
\section{Classifications of twisted superpotentials}
In this section, 
we will give complete lists of superpotentials 
and twisted superpotentials 
whose derivation-quotient algebras are 
$3$-dimensional quadratic AS-regular algebras except for Type EC, 
by using the following three steps:
\begin{description}
\item[Step I] (Proposition \ref{TSP}) Find the candidates of regular twisted superpotentials 
      corresponding to defining relations listed in \cite[Theorem 3.1]{IM}. 
\item[Step II] (Theorem \ref{SP}) Find all superpotentials among the above candidates 
      and show that they are Calabi-Yau superpotentials. 
\item[Step III] (Theorem \ref{prop-3}, Theorem \ref{thm-RTSP}) 
      Show that all above candidates can be 
      written as MS twists of Calabi-Yau superpotentials 
      and that they are in fact regular twisted superpotentials. 
\end{description}
As a byproduct, we will prove that, 
for any $3$-dimensional quadratic AS-regular algebra $A$ except for Type EC, 
there exist a Calabi-Yau AS-regular algebra $S$ 
and $\theta \in {\rm Aut}\,S$ 
such that $A$ is isomorphic to  $S^{\theta}$ as graded $k$-algebras. 
This result is needed to prove our main result Theorem \ref{conj}.  
\begin{prop}
\label{TSP}
Every $3$-dimensional quadratic AS-regular algebra except for Type EC 
is isomorphic to $\mathcal{D}(w)$ of a potential $w$ in 
Table $1$. 
\end{prop}

\noindent
{\small
\begin{tabular}{|p{0.6cm}|@{}p{2.9cm}|p{1.2cm}|p{3.2cm}|p{2.9cm}|}
\multicolumn{5}{c}{Table 1}\\[5pt]
\hline
& \ \ \ potential $w$ 
           & Cond. 
           & $\partial_{x}w$, $\partial_{y}w$, $\partial_{z}w$
           & $\nu_{A^{!}}$
           \\ \hline\hline
$\rm{P}_1$ & $\begin{array}{l}
             \alpha^2\beta xyz+\beta^2\gamma yzx \\ +\gamma^2\alpha zxy-\alpha^2\gamma xzy
             \\ -\gamma^2\beta zyx-\beta^2\alpha yxz
\end{array}$
           & $\alpha\beta\gamma\neq 0$
           & $\left\{
              \begin{array}{ll}
               \alpha^2\beta yz-\alpha^2\gamma zy, \\
               \beta^2\gamma zx-\beta^2\alpha xz, \\
               \gamma^2\alpha xy-\gamma^2\beta yx
              \end{array}
             \right.
             $
           & \rule{0pt}{40pt}$
             \left(
              \begin{array}{ccc}
              \dfrac{\alpha^{2}}{\beta\gamma} & 0& 0 \\
              0 & \dfrac{\beta^{2}}{\alpha\gamma}& 0 \\
              0 & 0& \dfrac{\gamma^{2}}{\alpha\beta}
             \end{array}
             \right)
             $ 
           \\[32pt] \hline
$\rm{P}_2$ & $\begin{array}{l}
             xyz+\alpha yzx+ \! \alpha^{2}zxy
             \\ -\alpha xzy-\alpha^{2}zyx \\ -yxz
             +y^{2}z \! - \! 2\alpha yzy \\ +\alpha^{2}zy^{2}
\end{array}$
           & $\alpha\neq 0$
           & $\left\{ \!\! 
              \begin{array}{ll}
                yz-\alpha zy, \\
                yz-2\alpha zy+\alpha zx \\ \hfill -xz,\  \\
                \alpha^{2}y^{2}+\alpha^{2}xy-\alpha^{2}yx
              \end{array}
             \right.
             $
             & \rule{0pt}{30pt}$
             \left(
              \begin{array}{ccc}
              \dfrac{1}{\alpha} & \dfrac{3}{\alpha}& 0 \\
              0 & \dfrac{1}{\alpha}& 0 \\
              0 & 0& \alpha^{2}
             \end{array}
             \right)
             $
           \\[20pt] \hline

$\rm{P}_3$ & $\begin{array}{l}
             -xyz-yzx-zxy \\ +xzy+zyx+yxz \\
             -z^{2}x+2zxz-xz^{2} \\ -zy^{2}+zyz+z^{2}y \\ 
             -y^{2}z+2yzy-2yz^{2} \\ -z^{3}
\end{array}$
           & nothing
           & \rule{0pt}{40pt}$\left\{ \!\! 
              \begin{array}{ll}
              zy-yz-z^{2}, \\
              xz-yz-2z^{2}\\
              \hfill -zx+2zy,\  \\
              -xy+yx-y^{2}-zx\\
              \hfill +2xz+yz+zy \\ \hfill -z^{2} 
              \end{array}
             \right.
             $
           & $
             \left(
              \begin{array}{ccc}
              1 & 3& 3 \\
              0 & 1& 3 \\
              0 & 0& 1
             \end{array}
             \right)
             $
            \\ \hline

           \end{tabular}
}

\noindent
{\small
\begin{tabular}{|p{0.6cm}|@{}p{2.9cm}|p{1.2cm}|p{3.2cm}|p{2.9cm}|}\hline
%
$\rm{S}_1$ & $\begin{array}{l}
             \beta xyz+\gamma yzx \! + \! \alpha zxy \\ -\alpha\beta xzy
             -\alpha\gamma zyx \\ -\beta\gamma yxz
\end{array}$
           & $\alpha\beta\gamma\neq 0,1$
           & $\left\{
              \begin{array}{ll}
               \beta yz-\alpha\beta zy, \\
               \gamma zx-\beta\gamma xz, \\
               \alpha xy-\alpha\gamma yx
              \end{array}
             \right.
             $
           & \rule{0pt}{35pt}$
             \left(
              \begin{array}{ccc}
              \dfrac{\beta}{\gamma} & 0& 0 \\
              0 & \dfrac{\gamma}{\alpha}& 0 \\
              0 & 0& \dfrac{\alpha}{\beta}
             \end{array}
             \right)
             $
           \\[25pt] \hline
$\rm{S}_2$ & $\begin{array}{l}
             -yzx-xzy+ \! \dfrac{1}{\beta}x^{2}z \\ +\dfrac{1}{\alpha}zx^{2}+\alpha y^{2}z
             \\ +\beta zy^{2}
\end{array}$
           & $\alpha\beta\neq 0$
           & $\left\{
              \begin{array}{ll}
              \dfrac{1}{\beta}xz-zy, \\
              \alpha yz-zx, \\
              \dfrac{1}{\alpha}x^{2}+\beta y^{2}
              \end{array}
             \right.
             $
           & \rule{0pt}{32pt}$
             \left(
              \begin{array}{ccc}
              0& -\alpha& 0 \\
              -\dfrac{1}{\beta} & 0& 0 \\
              0 & 0& \dfrac{\beta}{\alpha}
             \end{array}
             \right)
             $
            \\[25pt]
            \hline
$\rm{S}_3$ & $\begin{array}{l}
             -xzy-zyx-yxz
             \\ +\beta x^{3}+\gamma y^{3}+\alpha z^{3}
\end{array}$
           & $\alpha\beta\gamma\neq 0,1$
           & $\left\{
              \begin{array}{ll}
               \beta x^{2}-zy, \\
               \gamma y^{2}-xz, \\
               \alpha z^{2}-yx
              \end{array}
             \right.
             $
           & \rule{0pt}{22pt}$
             \left(
              \begin{array}{ccc}
              1& 0& 0 \\
              0& 1& 0 \\
              0 & 0& 1
             \end{array}
             \right)
             $
           \\[15pt] \hline
\rm{S'}$_1$ & $\begin{array}{l}
               \beta xyz+ \! \beta yzx \! + \! \alpha zxy
               \\ -\alpha\beta xzy-\alpha\beta zyx \\ -\beta^{2}yxz+\beta x^{3}
\end{array}$
            & $\alpha\beta^{2}\neq 0,1$
           & $\left\{
              \begin{array}{ll}
              \beta x^{2}+\beta yz-\alpha\beta zy, \\
              \beta zx-\beta^{2}xz, \\
              \alpha xy-\alpha\beta yx
              \end{array}
             \right.
             $
            & \rule{0pt}{30pt}$
              \left(
              \begin{array}{ccc}
              1& 0& 0 \\
              0& \dfrac{\beta}{\alpha}& 0 \\
              0 & 0& \dfrac{\alpha}{\beta}
              \end{array}
              \right)
              $
            \\[22pt] \hline
\rm{S'}$_2$ & $\begin{array}{l}
              -zxy-yxz+xy^{2} \\ +y^{2}x+xz^{2}+z^{2}x
              \\ +x^{3}
\end{array}$
            & nothing
           & $\left\{
              \begin{array}{ll}
               x^{2}+y^{2}+z^{2}, \\
               yx-xz, \\
               zx-xy
              \end{array}
             \right.
             $
            & \rule{0pt}{23pt}$
              \left(
              \begin{array}{ccc}
              1& 0& 0 \\
              0& 0& -1 \\
              0& -1& 0
              \end{array}
              \right)
              $
             \\[15pt] \hline
%
$\rm{T}_1$ & $\begin{array}{l}
             \rule{0pt}{12pt}
             \beta x^2y \\ +(\alpha-\beta+\gamma)xyx \\ 
             +(\alpha-\beta-\gamma)yxy \\ -\alpha y^2x-yxz+yzx
             \\ +\beta yx^2+xyz-xzy \\ -\alpha xy^2+zxy-zyx
\end{array}$
           & $\alpha+\beta+\gamma\neq 0$
           & $\left\{ \!\!
              \begin{array}{ll}
              \beta xy+(\alpha \! - \! \beta+\gamma)yx\\
              \hfill +yz-zy-\alpha y^2, \\
              (\alpha \! - \! \beta \! - \! \gamma)xy-\alpha yx\\
              \hfill -xz+zx+\beta x^2, \\
              xy-yx
              \end{array}
             \right.
             $
           & $
              \left(
              \begin{array}{ccc}
              1& 0& 0 \\
              0& 1& 0 \\              
              \delta& \varepsilon& 1
              \end{array}
              \right)
              $
              
             \rule{0pt}{12pt}
              $
              \delta:=-\alpha+2\beta-\gamma
              $, 
              
              $
              \varepsilon:=2\alpha-\beta-\gamma
              $
           \\ \hline
$\rm{T}_2$ & $\begin{array}{l}
\rule{0pt}{12pt}
              (1-\beta-\gamma)x^{3} \\ -(\alpha+2\gamma)yx^{2}+zx^{2}
              \\ -xy^{2}
              +\gamma y^{3}-zy^{2} \\ -x^{2}z+xzy+\beta x^{2}y
              \\ -y^{2}z+yzx+\alpha y^{2}x
\end{array}$
           & $\alpha+\beta+\gamma\neq 0$
           & $\left\{
              \begin{array}{ll}
              (1-\beta-\gamma)x^{2}-y^{2}\\
              \hfill -xz + zy+\beta xy, \\
             -(\alpha+2\gamma)x^{2}+\gamma y^{2}\\
             \hfill -yz+zx+\alpha yx,\\
             x^{2}-y^{2}
             \end{array}
             \right.
             $
           & $
              \left(
              \begin{array}{ccc}
              0& -1& 0 \\
              -1& 0& 0 \\
              \delta & \varepsilon &  -1
              \end{array}
              \right)
             $
             
             \rule{0pt}{12pt}
             $\delta:=-\beta+\gamma$
             
             $\varepsilon:= -\alpha+\gamma$
            \\ \hline
$\rm{T}_3$ & $\begin{array}{l}
\rule{0pt}{12pt}
             -x^3+y^3+x^2y \\ +xyx +yx^2-xy^2 \\ -yxy  -y^2x+x^2z \\ 
             +xzx 
             +zx^2+
             zy^2 \\ +yzy  +y^2z-xyz \\ -yzx  -zxy
\end{array}$
           & nothing
           & $
             \left\{
             \begin{array}{ll}
             -x^{2}+xy+yx-y^{2} \\
             \quad +xz+zx-yz,\\
             y^{2}+x^{2}-xy-yx\\
             \quad+zy+yz-zx, \\
             x^{2}+y^{2}-xy
             \end{array}
             \right.
             $
           & $
              \left(
              \begin{array}{ccc}
              1& 0& 0 \\
              0& 1& 0 \\
              0& 0& 1
              \end{array}
              \right)
             $
            \\ \hline         
\rm{T'} & $\begin{array}{l}
\rule{0pt}{12pt}
            \alpha x^2y- \! (\alpha-2\beta)xyx \\ 
            +(\beta^2-\alpha\beta)xy^2 \\
            +xyz -xzy+\alpha yx^2 \\ -yxz  +yzx-\alpha yzy
            \\ +\alpha \beta^2 y^3 \\ -(\beta^2-\alpha \beta)y^2x \\
            -\beta y^2z+zxy-zyx \\ -\beta zy^2
\end{array}$
            & $\alpha+2\beta\neq 0$
           & $
             \left\{
             \begin{array}{ll}
            \alpha xy-(\alpha-2\beta)yx\\
            \hfill +(\beta^{2}-\alpha\beta)y^{2} \\
            \hfill +yz-zy, \\
            \alpha x^{2}-xz+zx \\ 
            \hfill -\alpha zy+\alpha \beta^{2} y^{2}\\
            \hfill -(\beta^2-\alpha \beta)yx \\ 
            \hfill -\beta yz, \\
            xy-yx-\beta y^{2}
             \end{array}
             \right.
             $
            & $
              \left(
              \begin{array}{ccc}
              1& \delta& 0 \\
              0& 1& 0 \\
              2 \delta & \delta^{2}& 1
              \end{array}
              \right)
              $
              
               \rule{0pt}{12pt}
              $\delta:=\alpha-\beta$
            \\[22pt] \hline
$\rm{CC}$ & $\begin{array}{l}
\rule{0pt}{12pt}
            -3x^{3}-y^{2}x-yxy \\ -xy^{2}+y^{2}z+yzy \\ +zy^{2}
            -xyz-yzx \\ -zxy+xzy+zyx \\ +yxz
\end{array}$
          & nothing
           & $
             \left\{ \!\! \! 
             \begin{array}{ll}
            -3x^{2}-y^{2}-yz+zy, \\
            -yx-xy+yz+zy\\
            \hfill -zx+xz, \\
            y^{2}-xy+yx
             \end{array}
             \right.
             $
          & $
            \left(
            \begin{array}{ccc}
            1& 0& 0 \\
            0& 1& 0 \\
            0& 0& 1
            \end{array}
            \right)
            $
          \\ \hline

%

 \end{tabular}
}

\noindent
{\small
\begin{tabular}{|p{0.6cm}|@{}p{2.9cm}|p{1.2cm}|p{3.2cm}|p{2.9cm}|}\hline
%
%
$\rm{NC}_1$ & $\begin{array}{l}
              \dfrac{1 \! - \alpha^{3}}{\alpha} x^{3} \! + \! 
              \dfrac{1 \! -\alpha^{3}}{\alpha}y^{3} \\
              +xyz+yzx+zxy \\ -\alpha(xzy+zyx \\ \hfill +yxz)
\end{array}$
            & $\alpha^{3}\neq 0,1$
           & \rule{0pt}{42pt}$
             \left\{
             \begin{array}{ll}
              \dfrac{1-\alpha^{3}}{\alpha} x^{2}+yz \\
              \hfill -\alpha zy, \\
              \dfrac{1-\alpha^{3}}{\alpha}y^{2}+zx \\
              \hfill -\alpha xz, \\
              xy-\alpha yx
             \end{array}
             \right.
             $
            & $
              \left(
              \begin{array}{ccc}
              1& 0& 0 \\
              0& 1& 0 \\
              0& 0& 1
              \end{array}
              \right)
              $
           \\[35pt] \hline
$\rm{NC}_2$ & $\begin{array}{l}
\rule{0pt}{12pt}
              -2xyx+x^{2}z+ \! zx^{2} \\ -2yxy+y^{2}z+zy^{2} \\ +yzx+xzy
\end{array}$
            & nothing
           & $
             \left\{
             \begin{array}{ll}
              -2yx+xz+zy, \\
              -2xy+yz+zx,\\
              x^{2}+y^{2}
             \end{array}
             \right.
             $
            & $
              \left(
              \begin{array}{ccc}
              0& 1& 0 \\
              1& 0& 0 \\
              0& 0& 1
              \end{array}
              \right)
              $
            \\[15pt] \hline

$\rm{WL}_1$ & $\begin{array}{l}
 \rule{0pt}{12pt}
              -(1+\gamma)y^2x \\ +\alpha(1+2\gamma)yxy \\
              -\alpha^2(1+\gamma)xy^2 \\ +\alpha^2xyz+yzx \\
              +\alpha zxy-\alpha^2 xzy \\ -zyx-\alpha yxz
\end{array}$
            & $\alpha\neq 0,1$
           & $
             \left\{ \!\!\!
             \begin{array}{ll}
              -\alpha^{2}(1+\gamma)y^{2} \\
              \hfill +\alpha^{2}yz-\alpha^{2}zy, \\
              -(1+\gamma)yx \\ 
              \  +\alpha(1+2\gamma)xy +zx \\ 
              \hfill -\alpha xz, \\
              \alpha xy-yx
             \end{array}
             \right.
             $
            & $
              \left(
              \begin{array}{ccc}
              \alpha^{2}& 0& 0 \\
              0& \dfrac{1}{\alpha}& 0 \\
              0& \delta& \dfrac{1}{\alpha}
              \end{array}
              \right)
              $
              
              $\delta:=\dfrac{1}{\alpha}(2+3\gamma)$
              \\[8pt] \hline
%
$\rm{WL}_2$ & $\begin{array}{l}
              -(1+\gamma)y^{2}x \\
              +(1+2\gamma)yxy \\ 
              -(1+\gamma)xy^{2}+xyz \\ +yzx+zxy
              -xzy \\ -zyx-yxz
\end{array}$
            & nothing
           &  \rule{0pt}{40pt}$
             \left\{
             \begin{array}{ll}
              -(1+\gamma)y^{2}+yz \\
              \hfill -zy, \\
              -(1+\gamma)yx \\
              \ \  +(1+2\gamma)xy +zx \\
              \hfill -xz, \\
              xy-yx
             \end{array}
             \right.
             $
            & $
              \left(
              \begin{array}{ccc}
              1& 0& 0 \\
              0& 1& 0 \\
              0& 2+3\gamma& 1
              \end{array}
              \right)
              $
            \\[32pt] \hline
$\rm{WL}_3$ & $\begin{array}{l}
 \rule{0pt}{12pt}
              x^{2}y-2xyx+yx^{2} \\ -(1+\gamma)y^{2}x \\ +(1+2\gamma)yxy
              \\ -(1+\gamma)xy^{2}+xyz \\ +yzx+zxy-xzy \\ -zyx-yxz
\end{array}$
            & nothing
           & $
             \left\{ \!\!\!
             \begin{array}{ll}
              xy-2yx-(1+\gamma)y^{2} \\
              \hfill +yz-zy, \\
              x^{2}-(1+\gamma)yx \\
              \hfill +(1+2\gamma)xy +zx \\
              \hfill -xz, \\
              xy-yx
             \end{array}
             \right.
             $
            & $
              \left(
              \begin{array}{ccc}
              1& 0& 0 \\
              0& 1& 0 \\
              3& 2+3\gamma& 1
              \end{array}
              \right)
              $
              \\[33pt] \hline
%
$\rm{TL}_1$& $\begin{array}{l}
 \rule{0pt}{18pt}
             -\dfrac{1}{\alpha^{2}}zxy+\dfrac{1}{\alpha}zyx \\[8pt] +\alpha^{2}yxz
             \\ -\alpha yzx+\dfrac{1}{\alpha}xzy \\[8pt] -\alpha xyz -x^{3}
\end{array}$
           & $\alpha\neq 0$
           & $
             \left\{
             \begin{array}{ll}
             \dfrac{1}{\alpha}zy-\alpha yz -x^{2}, \\
             \alpha^{2}xz-\alpha zx, \\
             -\dfrac{1}{\alpha^{2}}xy+\dfrac{1}{\alpha}yx
             \end{array}
             \right.
             $
           & $
             \left(
             \begin{array}{ccc}
             1& 0& 0 \\
             0& \alpha^{3}& 0 \\
             0& 0& \dfrac{1}{\alpha^{3}}
             \end{array}
             \right)
             $
             \\[35pt] \hline
$\rm{TL}_2$& $\begin{array}{l}
             \beta x^2y+\beta xyx \\ 
             +(-\beta^2 \! - \! 1)x^3 \! \! + \! 2yxy \\ 
             -y^2x \! - \! 2\beta yx^2+zxy
             \\ -zyx \! - \! \beta zx^2 \! -\beta x^2z \\ 
             +2\beta xzx \! - \! yxz \! + \! yzx 
             \\ -xzy +xyz -xy^2
\end{array}$
           & nothing
           &  \rule{0pt}{45pt}$
             \left\{ \!\!
             \begin{array}{ll}
             \beta xy+\beta yx \\
             +(-\beta^2-1)x^{2}  -  \beta xz \\
             \hfill +2\beta zx-zy+yz  \\
             \hfill -y^{2}, \\
             2xy-yx-2\beta x^{2} \\
             \hfill -xz +zx, \\
             xy-yx-\beta x^{2}
             \end{array}
             \right.
             $
           & $
             \left(
             \begin{array}{ccc}
             1& 0& 0 \\
             3\beta& 1& 0 \\
             3\beta& 3& 1
             \end{array}
             \right)
             $
             \\[38pt] \hline
%
%
$\rm{TL}_3$& $\begin {array}{l}
             -2yxy-y^2x+zxy \\ +zyx-yxz-yzx \\ +xzy-xyz
             - x^3 \\
              -  xy^2
\end{array}$
           & nothing
           &  \rule{0pt}{28pt}$
             \left\{ \!
             \begin{array}{ll}
             zy-yz-x^{2}-y^{2}, \\
             -2xy-yx-xz \\
             \hfill -zx, \\
             xy+yx
             \end{array}
             \right.
             $
           & $
             \left(
             \begin{array}{ccc}
             1& 0& 0 \\
             0& -1& 0 \\
             0& 3& -1
             \end{array}
             \right)
             $
             \\[22pt] \hline
$\rm{TL}_4$& $\begin{array}{l}
 \rule{0pt}{12pt}
             -x^3+yx^2+x^2y \\ -2xyx+zxy+xyz \\ +yzx-zyx-yxz \\ -xzy
\end{array}$
           & nothing
           & $
             \left\{
             \begin{array}{ll}
             -x^{2}+xy-2yx \\
             \hfill +yz-zy, \\
             x^{2}+zx-xz,\\
             xy-yx
             \end{array}
             \right.
             $
           & $
             \left(
             \begin{array}{ccc}
             1& 0& 0 \\
             0& 1& 0 \\
             3& 0& 1
             \end{array}
             \right)
             $
             \\[22pt] \hline
\end{tabular}
}
\begin{proof}
All possible defining relations 
$f_{1},f_{2},f_{3}$ of $3$-dimensional quadratic AS-regular algebras 
except for Type EC were given in \cite[Theorem 3.1]{IM}. 
In each type, it is enough to find $w$ 
such that 
$(\partial_{x}{w},\partial_{y}{w},\partial_{z}{w})=(f_{1},f_{2},f_{3})$. 

We will give a proof for Type T$_{2}$ algebras. 
For the other types, the proofs are similar. 
From \cite[Theorem 3.1]{IM}, 
Type T$_{2}$ algebras are given as
$A=k\langle x,y,z \rangle/(f_{1},f_{2},f_{3})$: 
$$
\begin{cases}
f_{1}&=x^{2}-y^{2}, \\
f_{2}&=xz-
zy
-\beta xy+(\beta+\gamma)y^{2},\\
f_{3}&=yz-zx-\alpha yx+(\alpha+\gamma)x^{2}, 
\end{cases}
$$
where $\alpha+\beta+\gamma\neq 0$. 

Taking a potential
$$
\begin{array}{l}
w=(1-\beta-\gamma)x^{3}-(\alpha+2\gamma)yx^{2}+zx^{2}
     -xy^{2}+\gamma y^{3}-zy^{2}\\
      \quad\quad -x^{2}z+xzy+\beta x^{2}y-y^{2}z+yzx+\alpha y^{2}x
\end{array}
$$
as in Table 1, we have 
$$
\begin{cases}
\partial_{x}w&=(1-\beta-\gamma)x^{2}-y^{2}-xz+zy+\beta xy,\\
\partial_{y}w&=-(\alpha+2\gamma)x^{2}+\gamma y^{2}-yz+zx+\alpha yx,\\
\partial_{z}w&=x^{2}-y^{2}. 
\end{cases}
$$
Since $\partial_{z}w=x^{2}-y^{2}=f_{1}$, 
$\partial_{x}w=(1-\beta-\gamma)f_{1}-f_{2}$
and $\partial_{y}w=-\gamma f_{1}-f_{3}$, 
it follows that
$
\mathcal{D}(w)
=k\langle x,y,z\rangle/(\partial_{x}w, \partial_{y}w,\partial_{z}w)
=k\langle x,y,z\rangle/(f_{1}, f_{2}, f_{3})=A
$. 
\end{proof}
\begin{rem}
\label{rem_TL4}
(1)\quad
	For an algebra $A=k \langle x,y,z \rangle/(f_{1},f_{2},f_{3})$ of any 
type, 
	we can take a potential $w=xf_{1}+yf_{2}+zf_{3}$ such that $A=\mathcal{D}(w)$.
	But it is difficult to check that this potential 
        $w$ is a regular twisted superpotential
	and, in many cases, it is not so.
	The potentials $w$ listed in Proposition \ref{TSP} were chosen so that
	they are candidates of regular twisted superpotentials. 
By \cite[Theorem 4.4]{MS1}, 
every regular twisted superpotential $w$ satisfies
$(\nu^{-1}\otimes \mathrm{id}\otimes\mathrm{id})(\varphi(w))=w$
where $\nu$ is the Nakayama automorphism of $A^{!}$, so
in the above proposition,
we take a potential $w$ such that
$((\nu_{A^{!}})^{-1}\otimes \mathrm{id}\otimes\mathrm{id})(\varphi(w))=w$
where $\nu_{A^{!}}$ is the Nakayama automorphism of $A^!$
listed in Table $1$.

\noindent
(2)\quad 
For Type TL$_{4}$ in \cite[Theorem 3.1]{IM}, 
we gave the defining relations 
$$
\begin{cases}
&xy+yx,\\
&xz-zx-x^{2},\\
&zy-yz+xy+x^{2}.
\end{cases}
$$
The relation $xy+yx$ is a typo, 
and the relation $xy-yx$ is correct, 
which is given in Table 1 as above. 
\end{rem}
Next, we give a complete list of Calabi-Yau superpotentials as follows: 
\begin{thm}
\label{SP}
For every type except for Type EC, 
the following table is a complete list of Calabi-Yau superpotentials $w_{0}$. 
\end{thm}
\noindent
{\renewcommand\arraystretch{1.1}
{\small
\begin{tabular}{|c|p{4.5cm}|p{1.3cm}|p{4.6cm}|}
\multicolumn{4}{c}{Table 2}
\\[5pt]
\hline
& CY Superpotential $w_{0}$ 
           & Cond. 
           & $\partial_{x}w_{0}$, $\partial_{y}w_{0}$, $\partial_{z}w_{0}$
           \\ \hline\hline
$\rm{P}_1$ & 
            $\begin{array}{l}
             xyz+yzx+zxy \\ -\alpha(xzy+zyx+yxz)
\end{array}$
           & $\alpha^{3}=1$
           &  \rule{0pt}{25pt}$\left\{
             \begin{array}{ll}
             yz-\alpha zy,\\
             zx-\alpha xz,\\
             xy-\alpha yx
             \end{array}
             \right.
             $
           \\[18pt] \hline
$\rm{S}_1$ & $\begin{array}{l}
              xyz+ yzx+ zxy \\ -\alpha (xzy+zyx+yxz)
\end{array}$
           & $\alpha^{3}\neq 0,1$
           & \rule{0pt}{25pt}$\left\{
             \begin{array}{ll}
              yz-\alpha zy,\\
              zx-\alpha xz,\\
              xy-\alpha yx
             \end{array}
             \right.
             $
           \\[18pt] \hline
           %
$\rm{S}_3$ & $
             xzy+zyx+yxz-\alpha(x^{3}+y^{3}+z^{3})
             $
           & $\alpha^{3}\neq 0,1$
           & \rule{0pt}{25pt}$\left\{
             \begin{array}{ll}
             zy-\alpha x^{2},\\
             xz-\alpha y^{2},\\
             yx-\alpha z^{2}
             \end{array}
             \right.
             $
           \\[18pt] \hline
\rm{S'}$_1$ & $\begin{array}{l}
                xyz+ yzx+ zxy \\
               -\alpha (xzy+zyx+yxz)+x^{3}
\end{array}$
            & $\alpha^{3}\neq 0,1$
            & \rule{0pt}{25pt}$\left\{
             \begin{array}{ll}
               yz-\alpha zy+x^{2}, \\
               zx-\alpha xz, \\
               xy-\alpha yx
             \end{array}
             \right.
             $
            \\[18pt] \hline
$\rm{T}_1$ & $\begin{array}{l}
             xyz+yzx+zxy \\ -(xzy+zyx+yxz) \\ 
             +(x^2y+xyx+yx^2) \\
             -(y^2x+yxy+xy^2) 
\end{array}$
            & nothing
            & $\left\{
             \begin{array}{ll}
             yz-zy+xy+yx-y^{2}, \\
             zx-xz+x^{2}-yx-xy, \\
             xy-yx
             \end{array}
             \right.
             $
           \\ \hline
%
$\rm{T}_3$ & $\begin{array}{l}
             xyz+yzx+zxy \\ -(x^{2}y+xyx+yx^{2}) \\ 
             +(xy^{2}+yxy+y^{2}x) \\ -(x^{2}z+xzx+zx^{2}) \\ -
             (zy^{2}+yzy+y^{2}z)+x^{3}-y^{3}
\end{array}$
           & nothing
            & \rule{0pt}{35pt}$\left\{
             \begin{array}{ll}
             yz-xy-yx+y^{2}-xz-zx \\
             \hfill +x^{2}, \\
             zx-x^{2}+xy+yx-zy-yz \\ 
             \hfill -y^{2}, \\
             xy-x^{2}-y^{2}
             \end{array}
             \right.
             $
            \\[30pt] \hline
\rm{T'} & $\begin{array}{l}
             xyz+yzx+zxy \\ -(xzy+zyx+yxz) \\
            +(x^{2}y+xyx+yx^{2}) \\
            -(y^{2}z+yzy+zy^{2})+y^{3}
\end{array}$
            & nothing
            & $\left\{
             \begin{array}{ll}
            yz-zy+xy+yx, \\
            zx-xz+x^{2}-yz-zy+y^{2}, \\
            xy-yx-y^{2}
             \end{array}
             \right.
             $
            \\ \hline
$\rm{CC}$ & $\begin{array}{l}
            xyz+yzx+zxy \\ -(xzy+zyx+yxz) \\ 
            +(y^{2}x+yxy+xy^{2}) \\ -(y^{2}z+yzy+zy^{2})+3x^{3}
\end{array}$
          & nothing
          & $\left\{
             \begin{array}{ll}
            yz-zy+y^{2}+3x^{2}, \\
            zx-xz+yx+xy-yz-zy, \\
            xy-yx-y^{2}
             \end{array}
             \right.
             $
          \\ \hline
$\rm{NC}_1$ & $\begin{array}{l}
              xyz+yzx+zxy \\ -\alpha(xzy+zyx+yxz)
              +x^{3} {\!} + {\!} y^{3}
\end{array}$
            & $\alpha^{3}\neq 0,1$
            & \rule{0pt}{25pt}$\left\{
             \begin{array}{ll}
               yz-\alpha zy+x^{2}, \\
               zx-\alpha xz+y^{2}, \\
               xy-\alpha yx
             \end{array}
             \right.
             $
           \\[18pt] \hline

%
%
           \hline
$\rm{WL}_2$ & $\begin{array}{l}
              xyz+yzx+zxy \\ -(xzy+zyx+yxz) \\
              -\dfrac{1}{3}(y^{2}x+yxy+xy^{2})
\end{array}$\vspace{0.5em}
            & nothing
            & \rule{0pt}{35pt}$\left\{
             \begin{array}{ll}
             yz-zy-\dfrac{1}{3}y^{2},
             \vspace{0.5em} \\
             zx-xz-\dfrac{1}{3}(yx+xy), \\
             xy-yx
             \end{array}
             \right.
             $
            \\[25pt] \hline
$\rm{TL}_1$& $\begin{array}{l}
             xyz+yzx+zxy \\
             -\alpha(xzy+zyx+yxz)
              -x^{3}
\end{array}$
           & $\alpha^{3}=1$
            & \rule{0pt}{25pt}$\left\{
             \begin{array}{ll}
             yz-\alpha zy -x^{2},\\
             zx-\alpha xz,\\
             xy-\alpha yx
             \end{array}
             \right.
             $          
             \\[18pt] \hline
\end{tabular}
}}
\begin{proof}
Let $A$ be a $3$-dimensional quadratic AS-regular algebra 
except for Type EC.  
By Lemma \ref{thmRRZ}, 
$A$ is Calabi-Yau if and only if the Nakayama automorphism $\nu_{A^{!}}$ 
is 
the identity, 
where $A^{!}$ is the  quadratic dual of $A$. 
Considering the condition that $\nu_{A^{!}}$ in Table 1 is the identity, 
we have a superpotential, so
it is sufficient to show that $w_{0}$ is regular, that is, 
$\mathcal{D}(w_{0})=k\langle x,y,z \rangle /(\partial_{x}w_{0},\partial_{y}w_{0},\partial_{z}w_{0})$ 
is a $3$-dimensional quadratic AS-regular algebra. 
In fact, if $w$ is regular, then by Lemma \ref{CY-superpotential}, it is Calabi-Yau.
In order to prove AS-regularity of $\mathcal{D}(w_{0})$, 
we will check that $\mathcal{D}(w_{0})$ satisfies the conditions of Theorem \ref{ATV1_thm1}. 
Note that if $w$ is a superpotential, then the derivation-quotient algebra $\mathcal{D}(w)$
is standard if and only if the partial derivatives $\partial_{x}x$, $\partial_{y}w$, $\partial_{z}w$
are linearly independent (for example, see \cite[Proposition 2.6]{MS2}).

We will give a proof for Type T$_{1}$ algebras. 
For the other types, the proofs are similar. 
Let $w_{0}=xyz+yzx+zxy-(xzy+zyx+yxz) 
+(x^2y+xyx+yx^2)
-(y^2x+yxy+xy^2)$.
It is easy to check that $\partial_{x}w_{0},\partial_{y}w_{0},\partial_{z}w_{0}$
are linearly independent. 
For the potential $w_{0}$,
we have the unique $3 \times 3$ matrix
\begin{center}
$
M:=
\left(
\begin{array}{ccc}
y & x-y-z & y \\
x-y+z & -x & -x \\
-y & x & 0
\end{array}
\right)
$
such that $\left(
\begin{array}{c}
\partial_{x}w_{0} \\
\partial_{y}w_{0} \\
\partial_{z}w_{0}
\end{array}
\right)=M\left(
\begin{array}{c}
x \\
y \\
z
\end{array}
\right)$. 
\end{center}
By calculation, we have 
$
\left(
\begin{array}{ccc}
x & y & z
\end{array}
\right)
M=
\left(
\begin{array}{ccc}
\partial_{x}w_{0} &
\partial_{y}w_{0} &
\partial_{z}w_{0}
\end{array}
\right)
$. 
Hence, $\mathcal{D}(w_{0})$ is standard. 
We denote by $\Delta_{ij}$ the $(i,j)$-th $2\times 2$ minors 
of the matrix $M$ ($1\leq i,j \leq 3$). 
Since $\Delta_{11}=-x^{2}$, $\Delta_{22}=y^{2}$ and
$\Delta_{33}=-x^{2}+xy-y^{2}+z^{2}$,
we have that $\mathcal{V}(\{\Delta_{ij}\mid 1\leq i,j \leq 3\})=\emptyset$.

Therefore, by Theorem \ref{ATV1_thm1}, 
$\mathcal{D}(w_{0})$ is a $3$-dimensional quadratic AS-regular algebra, 
that is, 
$w_{0}$ is a Calabi-Yau superpotential. 
\end{proof}
\begin{thm}
\label{prop-3}
For a potential $w$ in Table $1$, 
there exist a Calabi-Yau superpotential $w_{0}$ in Table $2$ 
and $\theta \in {\rm Aut}\,(w_{0})$ such that 
$\mathcal{D}(w)\cong \mathcal{D}((w_{0})^{\theta})$
as in Table $3$.
\end{thm}

\noindent
\begin{tabular}{|p{0.8cm}|p{5cm}|p{5cm}|}
\multicolumn{3}{c}{Table 3}
\\[5pt]
\hline
\quad             & $w_{0}$ 
                       & $\theta$ \\ \hline\hline
\rm{P}$_{1}$ & $xyz+yzx+zxy-(xzy+zyx+yxz)$ 
                      &  
                      \rule{0pt}{25pt}$
                     \left(
                     \begin{array}{ccc}
                     \alpha & 0& 0 \\
                     0 & \beta & 0 \\
                    0 & 0& \gamma
                    \end{array}
                     \right)
                    $ 
                    \\[17pt] \hline
\rm{P}$_{2}$ & $xyz+yzx+zxy-(xzy+zyx+yxz)$ 
                      & \rule{0pt}{25pt}$
                        \left(
                        \begin{array}{ccc}
                        1 & 1& 0 \\
                        0 & 1 & 0 \\
                        0 & 0& \alpha
                        \end{array}
                        \right)
                        $\\[17pt] \hline
                        
                        \rm{P}$_{3}$ & $xyz+yzx+zxy-(xzy+zyx+yxz)$ 
                      & 
                      \rule{0pt}{25pt}$
                      \left(
                      \begin{array}{ccc}
                      1 & 1& 0 \\
                      0 & 1 & 1 \\
                      0 & 0& 1
                     \end{array}
                     \right)
                   $\\[17pt] \hline
\rm{S}$_{1}$ &  $(xyz+yzx+zxy)-\sqrt[3]{\alpha\beta\gamma}(xzy+yxz+zyx)$
                      &  \rule{0pt}{25pt}$
                          \left(
                          \begin{array}{c@{}c@{}c}
                          \sqrt[3]{\beta\gamma^{-1}}
                          & 0& 0 \\
                           0 & \sqrt[3]{\gamma\alpha^{-1}} & 0 \\
                          0 & 0& \sqrt[3]{\alpha\beta^{-1}}
                          \end{array}
                          \right)
                           $
                           \\[17pt] \hline
\rm{S}$_{2}$ & $(xyz+yzx+zxy)+(xzy+yxz+zyx)$
                      & \rule{0pt}{28pt}$
                         \left(
                         \begin{array}{c@{}c@{}c}
                         0 & -\sqrt[3]{\alpha^{2}\beta}& 0 \\
                         -\frac{1}{\sqrt[3]{\alpha\beta^{2}}} & 0 & 0 \\
                        0 & 0& \sqrt[3]{\beta\alpha^{-1}}
                        \end{array}
                        \right)
                        $
                        \\[19pt] \hline
\rm{S}$_{3}$ & $(xzy+yxz+zyx)-\alpha(x^{3}+y^{3}+z^{3})$
                      & \rule{0pt}{25pt}$
                         \left(
                         \begin{array}{ccc}
                          1 & 0& 0 \\
                          0 & 1 & 0 \\
                          0 & 0& 1
                          \end{array}
                           \right)
                          $
                          \\[17pt] \hline
\end{tabular}

\noindent
\begin{tabular}{|p{0.8cm}|p{5cm}|p{5cm}|}
\hline
\rm{S'}$_{1}$ &  $(xyz+yzx+zxy)-\sqrt[3]{\alpha\beta^{2}}(xzy+yxz+zyx)+x^{3}$
                       &  \rule{0pt}{25pt}$
                           \left(
                           \begin{array}{ccc}
                            1 & 0& 0 \\
                            0 & \sqrt[3]{\beta\alpha^{-1}} & 0 \\
                           0 & 0& \sqrt[3]{\alpha\beta^{-1}}
                           \end{array}
                            \right)
                           $ 
                           \\[17pt] \hline
\rm{S'}$_{2}$ & $(xyz+yzx+zxy)+(xzy+yxz+zyx)+x^{3}$
                      & 
                      \rule{0pt}{25pt}$
                      \left(
                      \begin{array}{ccc}
                      1 & 0& 0 \\
                      0 & 0 & 1 \\
                      0 & 1& 0
                      \end{array}
                      \right)
                       $
                        \\[17pt] \hline
%
\rm{T}$_{1}$ &  $(xyz+yzx+zxy)-(xzy+yxz+zyx)+(x^{2}y+xyx+yx^{2})-(xy^{2}+yxy+y^{2}x)$
                      &  \rule{0pt}{25pt}$
                          \left(
                          \begin{array}{ccc}
                            1 & 0& 0 \\
                            0 & 1 & 0 \\
                           \lambda\nu^{-1} & \mu\nu^{-1}& 1
                          \end{array}
                          \right)
                         $ 
                         
                         \rule{0pt}{12pt}$\lambda:=\frac{1}{3}(-\alpha+2\beta-\gamma)$,

                         \rule{0pt}{12pt}$\mu:=\frac{1}{3}(2\alpha-\beta-\gamma)$,

                         \rule{0pt}{12pt}$\nu:=\frac{1}{3}(\alpha+\beta+\gamma)$
                      \\[5pt] \hline
\rm{T}$_{2}$ & $(xyz+yzx+zxy)-(xzy+yxz+zyx)+(x^{2}y+xyx+yx^{2})-(xy^{2}+yxy+y^{2}x)$
                      & \rule{0pt}{25pt}$
                         \left(
                         \begin{array}{ccc}
                         0 & -1& 0 \\
                         -1 & 0 & 0 \\
                        -\lambda\nu^{-1} & -\mu\nu^{-1}& -1
                        \end{array}
                        \right)
                        $ 
                       
                       \rule{0pt}{12pt}$\lambda:=\frac{1}{3}(-\alpha+2\beta-\gamma)$,

                         \rule{0pt}{12pt}$\mu:=\frac{1}{3}(2\alpha-\beta-\gamma)$,

                         \rule{0pt}{12pt}$\nu:=\frac{1}{3}(\alpha+\beta+\gamma)$
                       \\[5pt] \hline
\rm{T}$_{3}$ &  \rule{0pt}{30pt}$\begin{array}{@{}l}(xyz+yzx+zxy)-(x^{2}y+xyx \\
+yx^{2})+(xy^{2}+yxy+y^{2}x) \\ -(x^{2}z+xzx+zx^{2})-(zy^{2} \\
+yzy+y^{2}z)+x^{3}-y^{3}\end{array}$
                      & $
                        \left(
                        \begin{array}{ccc}
                        1 & 0& 0 \\
                        0 & 1 & 0 \\
                        0 & 0& 1
                        \end{array}
                          \right)$
                          \\[21pt] \hline
\rm{T'} &  $(xyz+yzx+zxy)-(xzy+yxz+zyx)+(x^{2}y+xyx+yx^{2})-(y^{2}z+yzy+zy^{2})+y^{3}$
             &   \rule{0pt}{25pt}$
                 \left(
                 \begin{array}{ccc}
                 1 & \lambda^{-1}\mu& 0 \\
                 0 & 1 & 0 \\
                 2\lambda^{-1}\mu & \lambda^{-2}\mu^{2}& 1
                 \end{array}
                 \right)
                 $

                  \rule{0pt}{12pt}$\lambda:=\frac{1}{3}(\alpha+2\beta)$,

                  \rule{0pt}{12pt}$\mu:=\frac{1}{3}(\alpha-\beta)$
                  \\[5pt] \hline
%
\rm{CC} &  $\begin{array}{@{}l}(xyz+yzx+zxy)-(xzy+yxz \\
+zyx)+(y^{2}x+yxy+xy^{2}) \\
-(y^{2}z+yzy+zy^{2})+3x^{3}\end{array}$
               & 
                \rule{0pt}{25pt}$
               \left(
               \begin{array}{ccc}
                1 & 0& 0 \\
               0 & 1 & 0 \\
               0 & 0& 1
               \end{array}
               \right)
               $
               \\[18pt] \hline
%
\rm{NC}$_{1}$ & $(xyz+yzx+zxy)-\alpha(xzy+yxz+zyx)+x^{3}+y^{3}$
                      &   \rule{0pt}{25pt}$
                           \left(
                           \begin{array}{ccc}
                            1 & 0& 0 \\
                            0 & 1 & 0 \\
                            0 & 0& 1
                           \end{array}
                          \right)
                          $
                         \\[18pt] \hline
\rm{NC}$_{2}$ & $(xyz+yzx+zxy)+(xzy+yxz+zyx)+x^{3}+y^{3}$
                      &  \rule{0pt}{25pt}$
                        \left(
                        \begin{array}{ccc}
                        0 & 1& 0 \\
                        1 & 0 & 0 \\
                        0 & 0& 1
                      \end{array}
                       \right)
                      $
                      \\[18pt] \hline

\rm{WL}$_{1}$ &  $\begin{array}{@{}l}xyz+yzx+zxy-(xzy+zyx \\
+yxz)-\dfrac{1}{3}(y^{2}x+yxy+xy^{2})\end{array}$
                      &   \rule{0pt}{30pt}$
                          \left(
                          \begin{array}{ccc}
                          \alpha & 0& 0 \\
                           0 & 1 & 0 \\
                           0 & \dfrac{2}{3}+\gamma & 1
                          \end{array}
                          \right)
                        $                                            
                      \\[20pt] \hline
\end{tabular}
%


\noindent
\begin{tabular}{|p{0.8cm}|p{5cm}|p{5cm}|}
\hline
\rm{WL}$_{2}$ & $\begin{array}{@{}l}xyz+yzx+zxy-(xzy+zyx \\
+yxz)-\dfrac{1}{3}(y^{2}x+yxy+xy^{2})\end{array}$
                      &  \rule{0pt}{30pt}$
                          \left(
                         \begin{array}{ccc}
                         1 & 0& 0 \\
                         0 & 1 & 0 \\
                        0 & \dfrac{2}{3}+\gamma & 1
                       \end{array}
                       \right)
                      $
                       \\[20pt] \hline
\rm{WL}$_{3}$ & $\begin{array}{@{}l}xyz+yzx+zxy-(xzy+zyx \\
+yxz)-\dfrac{1}{3}(y^{2}x+yxy+xy^{2})\end{array}$
                      &  \rule{0pt}{30pt}$
                         \left(
                        \begin{array}{ccc}
                         1 & 0& 0 \\
                         0 & 1 & 0 \\
                          1 & \dfrac{2}{3}+\gamma& 1
                         \end{array}
                         \right)
                        $
                          \\[20pt] \hline

\rm{TL}$_{1}$ & $xyz+yzx+zxy-(xzy+zyx+yxz)-x^{3}$
                      &   \rule{0pt}{25pt}$
                          \left(
                          \begin{array}{ccc}
                          1 & 0& 0 \\
                          0 & \gamma & 0 \\
                          0 & 0& \gamma^{-1}
                         \end{array}
                        \right)
                        $  
                      \\[18pt] \hline
\rm{TL}$_{2}$ & $xyz+yzx+zxy-(xzy+zyx+yxz)-x^{3}$
                      &  \rule{0pt}{25pt}$
                         \left(
                         \begin{array}{ccc}
                         1 & 0& 0 \\
                         \beta & 1 & 0 \\
                         0 & 1& 1
                        \end{array}
                        \right)
                        $ 
                       \\[18pt] \hline
\rm{TL}$_{3}$ & $xyz+yzx+zxy-(xzy+zyx+yxz)-x^{3}$
                      &  \rule{0pt}{25pt}$
                          \left(
                          \begin{array}{ccc}
                          1 & 0& 0 \\
                          0 & -1 & 0 \\
                          0 & 1& -1
                          \end{array}
                          \right)
                           $
                        \\[18pt] \hline
\rm{TL}$_{4}$ & $xyz+yzx+zxy-(xzy+zyx+yxz)-x^{3}$
                      & 
                       \rule{0pt}{25pt}$
                       \left(
                       \begin{array}{ccc}
                       1 & 0& 0 \\
                       0 & 1 & 0 \\
                       1 & 0& 1
                       \end{array}
                       \right)
                       $
                        \\[18pt] \hline
\end{tabular}

\begin{proof}
	By direct computation, for a potential $w$ in Table 1, 
	we find a Calabi-Yau superpotential $w_{0}$ in Table 2 
	and $\theta \in {\rm Aut}\,(w_{0})$ such that 
	$\mathcal{D}(w)\cong\mathcal{D}((w_{0})^{\theta})$
	as in Table 3. 
	
	We will give a proof for 
	Type T$_{1}$ algebra. 
	For the other types, the proofs are similar. 
	Let $w$ be a potential of Type T$_{1}$ in Table $1$. 
	We take a superpotential 
	$w_{0}=xyz+yzx+zxy-(xzy+yxz+zyx)+(x^{2}y+xyx+yx^{2})-(xy^{2}+yxy+y^{2}x)$ and
	$\theta=
	\left(
	\begin{array}{ccc}
	1 & 0& 0 \\
	0 & 1 & 0 \\
	\lambda\nu^{-1} & \mu\nu^{-1}& 1
	\end{array}
	\right)
	$
	where 
	$\lambda:=\frac{1}{3}(-\alpha+2\beta-\gamma)$,
	$\mu:=\frac{1}{3}(2\alpha-\beta-\gamma)$ and
	$\nu:=\frac{1}{3}(\alpha+\beta+\gamma)$.
	Since $(\theta^{\otimes 3})(w_{0})=w_{0}$, 
	$\theta \in {\rm Aut}\,(w_{0})$.
	By calculation, we have that
	$(w_{0})^{\theta}=(xyz+yzx+zxy)-(xzy+yxz+zyx)
	+\nu^{-1}(\beta x^{2}y+(\alpha-\beta+\gamma)xyx+\beta yx^{2})
	-\nu^{-1}(\alpha y^{2}x+(-\alpha+\beta+\gamma)yxy+\alpha xy^{2})$.
	By taking
	$\theta'=
	\left(
	\begin{array}{ccc}
	1 & 0& 0 \\
	0 & 1 & 0 \\
	0 & 0 & \nu^{-1}
	\end{array}
	\right)
	\in {\rm GL}_{3}(k)$, it follows that $(\theta'^{\otimes^{3}})((w_{0})^{\theta})=\nu^{-1}w$, 
	so 
	$\mathcal{D}((w_{0})^{\theta}) \cong \mathcal{D}(w)$.
\end{proof}
We remark how we find $\theta$ in Table $3$. 
We take the third root of the matrix $(\nu_{A^!})^{-1}$, 
where $\nu_{A^!}$ is in Table $1$. 
Comparering the third root of the matrix $(\nu_{A^!})^{-1}$, 
we decide $\theta$ in Table $3$
(see Lemma \ref{lem-1} and Remark \ref{rem_TL4} (1)).

By Lemma \ref{lem-1}, Lemma \ref{lem-2}, Proposition \ref{TSP}, Theorem \ref{SP}  and 
Theorem \ref{prop-3}, 
the following theorem immediately holds. 
\begin{thm}
\label{thm-RTSP}
Any potential $w$ in Table $1$ is a regular twisted superpotential. 
\end{thm}
\begin{rem}
\label{rem_reg}
It turns out from Theorem \ref{thm-RTSP} that 
the defining relations listed in \cite[Theorem 3.1]{IM} 
are in fact those of $3$-dimensional quadratic AS-regular algebras 
(see also Remark \ref{rem_IM1}). 
\end{rem}
By Theorem \ref{prop-3},
for a $3$-dimensional AS-regular algebra $A$ except for Type EC,
there exist a Calabi-Yau superpotential $w_{0}$ 
and $\theta \in {\rm Aut}(w_{0})$ in Table $3$
such that $A \cong \mathcal{D}((w_{0})^{\theta})$.
Since $\theta \in {\rm Aut}(w_{0})$, we have that
$\mathcal{D}((w_{0})^{\theta}) \cong \mathcal{D}(w_{0})^{\theta}$
by Lemma \ref{twist}.
Since $\mathcal{D}(w_{0})$ is Calabi-Yau AS-regular,
we have the following corollary:
\begin{cor}
\label{cor_conj}
For a $3$-dimensional quadratic AS-regular algebra 
$A$ except for Type EC, 
there exist a Calabi-Yau AS-regular algebra $S$ 
and $\theta \in {\rm Aut}\,S$ 
such that $A$ is isomorphic to  $S^{\theta}$ as graded $k$-algebras. 
\end{cor}
\section{Geometric algebras of Type EC}
We say that a geometric algebra $A=\mathcal{A}(E,\sigma)$ is 
{\it of Type EC} 
if $E$ is an elliptic curve in $\mathbb{P}^{2}$.
In this section, we give a criterion when
a geometric algebra of Type EC is a $3$-dimensional quadratic
AS-regular algebra.
\subsection{Divisors on curves and Hesse forms}
Let $E$ be a projective smooth curve over $k$.
The {\it Picard group} of $E$, 
denoted by $({\rm Pic}\,E,[\mathcal{O}_{E}],\otimes)$,
is the group of isomorphism classes of
invertible sheaves on $E$ under the operation $\otimes$ 
(see \cite[page 143]{H}).

A {\it divisor} on $E$ is an element of the free abelian group 
$$
{\rm Div}\,E:=\left\{ \sum_{p \in E}n_{p}\cdot p\,\middle |\,n_{p} \in \mathbb{Z} \right\}
$$
where only finitely many $n_{p}$ are different from zero.
We write the group of divisors $({\rm Div}\,E,0,+)$
where $0$ is the zero divisor, that is, $n_{p}=0$ for any $p \in E$.
For any divisor $D \in {\rm Div}\,E$,
there exists an invertible sheaf on $E$, denoted by $\mathcal{O}_{E}(D)$,
and the map $D \mapsto \mathcal{O}_{E}(D)$ gives a surjective homomorphism
from $({\rm Div}\,E,0,+)$ to $({\rm Pic}\,E,[\mathcal{O}_{E}],\otimes)$
(see  \cite[Proposition II 6.13]{H} and \cite[Corollary II 6.16]{H}),
that is, for any $[\mathcal{M}] \in {\rm Pic}\,E$ 
there exists a divisor $D \in {\rm Div}\,E$
such that $[\mathcal{M}]=[\mathcal{O}_{E}(D)]$.
The zero divisor $0$ maps to 
the isomorphism class $[\mathcal{O}_{E}] \in {\rm Pic}\,E$.

For $\sigma \in {\rm Aut}_{k}\,E$,
we define a map $\tilde{\sigma}:\Div E \rightarrow \Div E$ by
$$
\tilde{\sigma}\left( \sum_{p \in E}n_{p} \cdot p \right)
=\sum_{p \in E}n_{p} \cdot \sigma^{-1}(p).
$$
This map $\tilde{\sigma}$ is a group automorphism
of $(\Div E,0,+)$. 
On the other hand, for $\sigma \in {\rm Aut}_{k}\,E$,
the rule $\mathcal{M} \mapsto \sigma^{\ast}\mathcal{M}$
where $\mathcal{M}$ is an invertible sheaf on $E$
induces a group automorphism of the Picard group $\sigma^{\ast}:\Pic E \rightarrow \Pic E$.
It follows from \cite[II Ex. 6.8]{H} that, if $D \in \Div E$, then
$$
\sigma^{\ast}(\mathcal{O}_{E}(D)) \cong \mathcal{O}_{E}(\tilde{\sigma}D).
$$

Let $E$ be an elliptic curve in $\mathbb{P}^{2}$.
It is well-known that the {\it $j$-invariant} $j(E)$ classifies 
elliptic curves up to isomorphism, that is,
two elliptic curves $E$ and $E'$ in $\mathbb{P}^{2}$
are isomorphic if and only if $j(E)=j(E')$
(see \cite[Theorem IV 4.1(b)]{H}).
For $p \in E$, we define
${\rm Aut}_{k}\,(E,p):=\{ \sigma \in {\rm Aut}_{k}\,E \mid \sigma(p)=p \}$.
It follows from \cite[Corollary IV 4.7]{H} that, for every point $p \in E$,
${\rm Aut}_{k}\,(E,p)$ becomes a cyclic group of order 
$$
|{\rm Aut}_{k}\,(E,p)|
=\begin{cases}
2 &\quad\quad \text{if } j(E) \neq 0, 12^{3}, \\
6 &\quad\quad \text{if } j(E)=0, \\
4 &\quad\quad \text{if } j(E)=12^{3}.
\end{cases}
$$
For each point $o \in E$, we can define an addition $\oplus$ on $E$
so that $(E,o,\oplus)$ is an abelian group with the zero element $o$
and, for $p \in E$, the map $\sigma_{p}$ defined by
$\sigma_{p}(q):=p \oplus q$ is a scheme automorphism of $E$,
called the {\it translation} by a point $p$.

\subsection{Type EC}
Throughout this subsection, 
for an elliptic curve $E$ in $\mathbb{P}^{2}$, 
we use a {\it Hesse form}
$E=\mathcal{V}(x^{3}+y^{3}+z^{3}-3\la xyz)$
where $\la \in k$ with $\la^{3} \neq 1$.
The {\it $j$-invariant} of a Hesse form is given by the following formula
(see \cite[Proposition 2.16]{F}): 
$$
j(E)=\frac{27\la^{3}(\la^{3}+8)^{3}}{(\la^{3}-1)^{3}}. 
$$

We fix the group structure on $E$ with the zero element
$o_{E}:=(1:-1:0) \in E$.
Every automorphism $\sigma \in {\rm Aut}_{k}\,E$
can be written as $\sigma=\sigma_{p}\tau^{i}$
where $p \in E$, $\tau$ is a generator of ${\rm Aut}_{k}\,(E,o_{E})$
and $i \in \mathbb{Z}_{|\tau|}$ (\cite[Proposition 4.5 and Theorem 4.6]{IM}).

We call a point $p \in E$ {\it $3$-torsion} if $p \oplus p \oplus p=o_{E}$.
We set $E[3]:=\{ p \in E \mid p \oplus p \oplus p=o_{E} \}$.
For $p\in E$ and $i\in \mathbb{Z}$,
$A=\mathcal{A}(E,\sigma_{p}\tau^{i})$ is of Type EC 
if and only if $p \in E \setminus E[3]$ (\cite[Lemma 4.14]{IM}).

The map $p \mapsto [\mathcal{O}_{E}(p-o_{E})]$ is an injective homomorphism
from $(E,o_{E},\oplus)$ to $(\Pic E,[\mathcal{O}_{E}],\otimes)$
(see \cite[Example IV 1.3.7]{H}).
For $p \in E$ and $n \in \mathbb{Z}$, 
we use a notation
$[n]p:=\underbrace{p \oplus \cdots \oplus p}_{n}$. 
It is easy to check the following lemma.
\begin{lem}\label{sum}
	Let $(E,o_{E},\oplus)$ be an elliptic curve in $\mathbb{P}^{2}$,
	$p \in E$ and $n \in \mathbb{Z}$. Then
	$$
	[\mathcal{O}_{E}([n]p-o_{E})]=[\mathcal{O}_{E}(n(p-o_{E}))].
	$$
\end{lem}
%

Since the zero element $o_{E}=(1:-1:0)$ is an inflection point of $E$,
it follows that
$\mathcal{L} \cong \mathcal{O}_{E}(3o_{E})$
where $\mathcal{L}=\pi^{\ast}(\mathcal{O}_{\mathbb{P}^{2}}(1))$.
\begin{lem}[cf. {\rm \cite[Lemma 4.5]{Mo}}]\label{Lemma1}
	Let $\pi:E \rightarrow \mathbb{P}^{2}$ 
	be the embedding and $\mathcal{L}=\pi^{\ast}\mathcal{O}_{\mathbb{P}^{2}}(1)$.
	Then an automorphism $\sigma \in {\rm Aut}_{k}E$ can be extended to an automorphism
	of $\mathbb{P}^{2}$ if and only if $\sigma^{\ast}\mathcal{L} \cong \mathcal{L}$.
\end{lem}
A $3$-dimensional {\it Sklyanin algebra} is defined by 
$\mathcal{A}(E,\sigma_{p})$ where $p=(a:b:c)\in E \setminus E[3]$
and the defining relations are given as follows: 
$$
\begin{cases}
ayz+bzy+cx^{2}, \\
azx+bxz+cy^{2}, \\
axy+byx+cz^{2}. 
\end{cases}
$$
A $3$-dimensional Sklyanin algebra $\mathcal{A}(E,\sigma_{p})$
is a $3$-dimensional quadratic AS-regular algebra by \cite[Section 1]{ATV1}.
It follows from \cite[Theorem 4.12 (1)]{IM} that
$\sigma_{p}\tau^{i}$ is not extended to an automorphism of $\mathbb{P}^{2}$,
so, by Lemma \ref{Lemma1},
$$
(\sigma_{p}\tau^{i})^{\ast}\mathcal{L} \not\cong \mathcal{L}.
$$

For a geometric algebra $A=\mathcal{A}(E,\sigma)$ of Type EC, 
we give a criterion when $A$ is AS-regular. 
\begin{thm}\label{Main}
	Let $A=\mathcal{A}(E,\sigma)$ be a geometric algebra of Type EC 
	where $\sigma=\sigma_{p}\tau^{i}$, $p \in E \backslash E[3]$ and $i \in \mathbb{Z}_{|\tau|}$. 
	Then the following are equivalent{\rm :}
	\begin{enumerate}[{\rm (1)}]
		\item $A$ is a $3$-dimensional quadratic AS-regular algebra.
		\item $p \ominus \tau^{i}(p) \in E[3]$.
		\item $A$ is graded Morita equivalent to
		a $3$-dimensional Sklyanin algebra $\mathcal{A}(E,\sigma_{p})$.
	\end{enumerate}
\end{thm}

\begin{proof}
	(1) $\Rightarrow$ (2):
	Assume that $A=\mathcal{A}(E,\sigma)$ is a $3$-dimensional quadratic AS-regular algebra.
       By Theorem \ref{ATV1}, it holds that
	$$
	(\sigma^{2})^{\ast}\mathcal{L} \otimes_{\mathcal{O}_{E}} \mathcal{L} \cong
	\sigma^{\ast}\mathcal{L} \otimes_{\mathcal{O}_{E}} \sigma^{\ast}\mathcal{L}.
	$$
	Since $\mathcal{L} \cong \mathcal{O}_{E}(3o_{E})$,
	\begin{align*}
	&\sigma^{\ast}\mathcal{L} \cong \sigma^{\ast}(\mathcal{O}_{E}(3o_{E}))
	\cong \mathcal{O}_{E}
	(\widetilde{\sigma}
	(3o_{E}))
	=\mathcal{O}_{E}(3q), \\
	&(\sigma^{2})^{\ast}\mathcal{L} \cong (\sigma^{2})^{\ast}(\mathcal{O}_{E}(3o_{E}))
	\cong \mathcal{O}_{E}
	(\widetilde{\sigma^{2}}
	(3o_{E}))
	=\mathcal{O}_{E}(3r),
	\end{align*}
	where $q:=\sigma^{-1}(o_{E})=\ominus\tau^{-i}(p)$ and
	$r:=\sigma^{-2}(o_{E})=\sigma^{-1}(q)=q \oplus \tau^{-i}(q)$, so
	\begin{align*}
	&(\sigma^{2})^{\ast}\mathcal{L} \otimes_{\mathcal{O}_{E}} \mathcal{L}
	\cong \mathcal{O}_{E}(3r) \otimes_{\mathcal{O}_{E}} \mathcal{O}_{E}(3o_{E})
	\cong \mathcal{O}_{E}(3(r+o_{E})), \\
	&\sigma^{\ast}\mathcal{L} \otimes_{\mathcal{O}_{E}} \sigma^{\ast}\mathcal{L}
	\cong \mathcal{O}_{E}(3q) \otimes_{\mathcal{O}_{E}} \mathcal{O}_{E}(3q)
	\cong \mathcal{O}_{E}(6q).
	\end{align*}
	Therefore,
	\begin{align*}
	(\sigma^{2})^{\ast}\mathcal{L} \otimes_{\mathcal{O}_{E}} \mathcal{L} \cong
	\sigma^{\ast}\mathcal{L} \otimes_{\mathcal{O}_{E}} \sigma^{\ast}\mathcal{L}
	&\Longrightarrow
	\mathcal{O}_{E}(3(r+o_{E})) \cong \mathcal{O}_{E}(6q) \\
	&\Longrightarrow
	\mathcal{O}_{E}(3(r-o_{E})) \cong \mathcal{O}_{E}(6(q-o_{E})) \\
	&\Longrightarrow
	\mathcal{O}_{E}([3]r-o_{E}) \cong \mathcal{O}_{E}([6]q-o_{E}) \\
	&\Longrightarrow
	[3]r=[6]q \\
	&\Longrightarrow
	[3](q \ominus \tau^{-i}(q))=o_{E} \\
	&\Longrightarrow
	q \ominus \tau^{-i}(q) \in E[3].
	\end{align*}
	Since $q=\ominus \tau^{-i}(p)$,
	$
	q \ominus \tau^{-i}(q)=\ominus \tau^{-i}(p) \oplus \tau^{-2i}(p)=\tau^{-2i}(p \ominus \tau^{i}(p))
	$. 
	Hence we have
	$
	p \ominus \tau^{i}(p) \in E[3]
	$. 
	
	(2) $\Rightarrow$ (3):
	Assume that $p \ominus \tau^{i}(p) \in E[3]$.
	By \cite[Theorem 4.20]{IM}, $A=\mathcal{A}(E,\sigma_{p}\tau^{i})$ and $\mathcal{A}(E,\sigma_{p})$
	are graded Morita equivalent.
	
	(3) $\Rightarrow$ (1):
	Assume that $A=\mathcal{A}(E,\sigma_{p}\tau^{i})$ is graded Morita equivalent to
	a $3$-dimensional Sklyanin algebra $\mathcal{A}(E,\sigma_{p})$.
	By Lemma \ref{twistingequi}, $A$ is isomorphic to a twisted algebra of $\mathcal{A}(E,\sigma_{p})$.
	Since being AS-regular is invariant under twisting system by Lemma \ref{twist_inv},
	a twisted algebra of $\mathcal{A}(E,\sigma_{p})$ is a $3$-dimensional quadratic AS-regular algebra.
	Therefore, $A=\mathcal{A}(E,\sigma_{p}\tau^{i})$ is a $3$-dimensional quadratic AS-regular algebra.
\end{proof}
Now, we are ready to prove Theorem \ref{conj} 
in 
the 
 Introduction. 
\begin{thm}
\label{Main2}
For every $3$-dimensional quadratic AS-regular algebra $A$, 
   there exists a Calabi-Yau AS-regular algebra $S$ 
  such that $A$ and $S$ are graded Morita equivalent. 
\end{thm}
\begin{proof}
      \underline{Except for Type EC}: 
      By Lemma \ref{twistingequi} and Corollary \ref{cor_conj}, 
      the statement holds. 
      
      \underline{For Type EC}:
      Let $A=\mathcal{A}(E,\sigma_{p})=\mathcal{D}(w)$ be a $3$-dimensional Sklyanin algebra
      where $p=(a:b:c) \in E \setminus E[3]$ and
      $w=a(xyz+yzx+zxy)+b(xzy+zyx+yxz)+c(x^{3}+y^{3}+z^{3})$.
      Since $w$ is a superpotential,
      by \ref{CY-superpotential}, 
      $\mathcal{D}(w)$ is Calabi-Yau AS-regular. 
      By Theorem \ref{Main}, the statement holds. 
\end{proof}
In Section 3, Corollary \ref{cor_conj} tells us that, 
for a $3$-dimensional AS-regular algebra except for Type EC, 
there exist a Calabi-Yau AS-regular algebra $S$ 
and $\theta \in {\rm Aut}\,S$ 
such that $A$ is isomorphic to  $S^{\theta}$ as graded $k$-algebras. 
We prove this by using Theorem \ref{prop-3}, 
that is, for  a potential $w$ in Theorem \ref{TSP}, 
there exist a superpotential $w_{0}$ in Theorem \ref{SP} and 
$\theta \in {\rm Aut}\,(w_{0})$ such that 
	$\mathcal{D}(w)\cong\mathcal{D}((w_{0})^{\theta})$.
On the other hand, 
it follows from \cite[Theorem 4.9]{IM} that, 
for a potential $w$ of a geometric algebra of Type EC, 
there exist a Calabi-Yau superpotential $w_{0}$  and 
$\theta \in {\rm GL}\,(V)$ induced by $\tau^{i}\in {\rm Aut}_{k}\,E$ 
such that $w=(w_{0})^{\theta}$.
But this $\theta$ is not necessary in ${\rm Aut}\,(w_{0})$, 
so $w=(w_{0})^{\theta}$ may not be a twisted superpotential nor regular 
(see Example \ref{examEC}).
This means that we do not know whether Theorem \ref{prop-3} holds or not
for Type EC. 
So, we need to divide the proof of Theorem \ref{conj} into two cases of 
non Type EC and Type EC. 
\begin{exa}
	Let $E=\mathcal{V}(x^{3}+y^{3}+z^{3}-3 \la xyz)$ be an elliptic curve in $\mathbb{P}^{2}$
	with $j(E) \neq 0,12^{3}$ and $A=\mathcal{A}(E,\sigma_{p}\tau)$ where $p \in E \setminus E[3]$.
	In this case, we have that $\tau(p)=\ominus p$. By Theorem \ref{Main},
	$A$ is a $3$-dimensional quadratic AS-regular algebra if and only if $[2]p \in E[3]$, that is, $p \in E[6]$
	where $E[6]:=\{ q \in E \mid [6]q=o_{E} \}$.
\end{exa}

\begin{exa}\label{examEC}
	In general,
	it is not true that if $w$ is a regular superpotential and
	$\theta \in {\rm GL}(V)$, then the MS twist $w^{\theta}$ is regular. 
	Let $E=\mathcal{V}(x^{3}+y^{3}+z^{3}-3\lambda xyz)$ be an elliptic curve with $j(E) \neq 0,12^{3}$
	and $A=\mathcal{A}(E,\sigma_{p}\tau)$ where $p=(a:b:c) \in E \setminus E[6]$.
	By Theorem \ref{Main}, $A$ is not AS-regular. 
	By \cite[Theorem 4.6]{IM} and \cite[Theorem 4.9]{IM},
	we have that $A=\mathcal{A}(E,\sigma_{p}\tau)=\mathcal{D}(w^{\theta})$
	where 
	$w=a(xyz+yzx+zxy)+b(xzy+yxz+zyx)+c(x^{3}+y^{3}+z^{3})$ and
	$\theta= \left(
	\begin{array}{ccc}
	0& 1& 0 \\
	1& 0& 0 \\
	0& 0& 1
	\end{array}
	\right) \in {\rm GL}_{3}(k)$.
	Note that if $a \neq b$, then $\theta \notin {\rm Aut}\,(w)$.
Since $A$ is not a $3$-dimensional quadratic AS-regular algebra,
$w^{\theta}$ is not regular.
\end{exa}
\proof[Acknowledgements]
The authors thank the referee for helpful comments 
in improving the paper.
They are also grateful to Professor Izuru Mori 
for his support and helpful discussions. 
Moreover, they appreciate Shinichi Hasegawa and Kosuke Shima 
for their helping to build and to check 
Table 1 in Proposition \ref{TSP}.
For Remark \ref{rem_TL4} (2), the authors thank 
Professor Andrew Conner and Professor Peter Goetz 
for telling them a typo for the relation of Type TL$_{4}$. 
The first author was supported by 
Grants-in-Aid for Young Scientific Research 18K13397 
Japan Society for the Promotion of Science. 

\end{document}